\documentclass[a4paper]{amsart}
\usepackage{amssymb, enumitem}
\usepackage[all]{xy}
\usepackage{hyperref, aliascnt, graphicx}
\usepackage{mathtools}

\usepackage{xcolor}

\newcommand{\andSep}{\,\,\,\text{ and }\,\,\,}
\newcommand{\axiomO}[1]{(O#1)}
\newcommand{\CatCu}{\ensuremath{\mathrm{Cu}}}
\newcommand{\CuSgp}{$\CatCu$-sem\-i\-group}

\newcommand{\soft}{{\rm{soft}}}

\newtheorem{lma}{Lemma}[section]

\newaliascnt{thmCt}{lma}
\newtheorem{thm}[thmCt]{Theorem}
\aliascntresetthe{thmCt}

\newaliascnt{corCt}{lma}
\newtheorem{cor}[corCt]{Corollary}
\aliascntresetthe{corCt}

\newaliascnt{prpCt}{lma}
\newtheorem{prp}[prpCt]{Proposition}
\aliascntresetthe{prpCt}

\newtheorem*{thm*}{Theorem}
\newtheorem*{cor*}{Corollary}
\newtheorem*{prop*}{Proposition}

\theoremstyle{definition}

\newaliascnt{pgrCt}{lma}
\newtheorem{pgr}[pgrCt]{}
\aliascntresetthe{pgrCt}

\newaliascnt{dfnCt}{lma}
\newtheorem{dfn}[dfnCt]{Definition}
\aliascntresetthe{dfnCt}

\newaliascnt{rmkCt}{lma}
\newtheorem{rmk}[rmkCt]{Remark}
\aliascntresetthe{rmkCt}

\newaliascnt{rmksCt}{lma}

\aliascntresetthe{rmksCt}

\newaliascnt{exaCt}{lma}
\newtheorem{exa}[exaCt]{Example}
\aliascntresetthe{exaCt}

\newaliascnt{egCt}{lma}

\aliascntresetthe{egCt}

\newaliascnt{qstCt}{lma}
\newtheorem{qst}[qstCt]{Question}
\aliascntresetthe{qstCt}

\newaliascnt{pbmCt}{lma}

\aliascntresetthe{pbmCt}

\newaliascnt{cnjCt}{lma}

\aliascntresetthe{cnjCt}

\newaliascnt{ntnCt}{lma}

\aliascntresetthe{ntnCt}

\def\today{\number\day\space\ifcase\month\or   January\or February\or
   March\or April\or May\or June\or   July\or August\or September\or
   October\or November\or December\fi\   \number\year}

\newcommand{\vertiii}[1]{{\left\vert\kern-0.25ex\left\vert\kern-0.25ex\left\vert #1
\right\vert\kern-0.25ex\right\vert\kern-0.25ex\right\vert}}

\newcommand{\NN}{{\mathbb{N}}}

\newcommand{\CC}{{\mathbb{C}}}

\newcommand{\NNbar}{\overline{\NN}}

\DeclareMathOperator{\Cu}{Cu}
\DeclareMathOperator{\wdiv}{div}
\DeclareMathOperator{\mwdiv}{mdiv}
\DeclareMathOperator{\sdiv}{Div}

\newcommand{\ca}{$\mathrm{C}^*$-algebra}

\newcommand{\stHom}{${}^*$-homomorphism}

\newcounter{theoremintro}

\newaliascnt{thmIntroCt}{theoremintro}

\aliascntresetthe{thmIntroCt}

\RequirePackage[normalem]{ulem} 
\RequirePackage{color}\definecolor{RED}{rgb}{1,0,0}\definecolor{BLUE}{rgb}{0,0,1} 
\RequirePackage{listings} 
\RequirePackage{color} 
\lstdefinelanguage{DIFcode}{ 
  moredelim=[il][\color{red}\sout]{\%DIF\ <\ }, 
  moredelim=[il][\color{blue}\uwave]{\%DIF\ >\ } 
} 
\lstdefinestyle{DIFverbatimstyle}{ 
	language=DIFcode, 
	basicstyle=\ttfamily, 
	columns=fullflexible, 
	keepspaces=true 
} 
\lstnewenvironment{DIFverbatim}{\lstset{style=DIFverbatimstyle}}{} 
\lstnewenvironment{DIFverbatim*}{\lstset{style=DIFverbatimstyle,showspaces=true}}{} 

\title{Nowhere scattered multiplier algebras}
\date{\today}

\author{Eduard Vilalta}
\address{Eduard Vilalta,
Departament de Matem\`{a}tiques,
Universitat Aut\`{o}noma de Barcelona,
08193 Bellaterra, Barcelona, Spain.}
\email{evilalta@mat.uab.cat}
\urladdr{www.eduardvilalta.com}

\thanks{The author was partially supported by MINECO (grant No.\ PID2020-113047GB-I00  and No.\ PRE2018-083419), and by the Departament de Recerca i Universitats de la Generalitat de Catalunya (grant No.\ 2021-SGR-01015).}

\begin{document}


\maketitle
\begin{abstract}
 We study sufficient conditions under which a nowhere scattered \ca{} $A$ has a nowhere scattered multiplier algebra $\mathcal{M}(A)$, that is, we study when $\mathcal{M}(A)$ has no nonzero, elementary ideal-quotients. In particular, we prove that a $\sigma$-unital \ca{} $A$ of
 \begin{itemize}
  \item[(i)] finite nuclear dimension, or
  \item[(ii)] real rank zero, or
  \item[(iii)] stable rank one with $k$-comparison,
 \end{itemize}
 is nowhere scattered if and only if $\mathcal{M}(A)$ is.
\end{abstract}

\section{Introduction}

The study of regularity properties of  multiplier algebras appears throughout the literature; see, for example, \cite{Ell74:DevMatAlg}, \cite{Zha90RieszDecomp}, \cite{Ror91IdMult}, and \cite{KafNgZha17StrCompMultiplier}. One notable instance of this is the study of pure infiniteness and, more concretely, of when a \ca{} has a purely infinite multiplier algebra; see \cite{LinSimpleCorAlg}, \cite{LinSimplePurInfCorAlg}, \cite{KurNgPer10MA}, and \cite{KafNgZha19PurInfMultAlg}. When this condition is relaxed to weak pure infiniteness (in the sense of \cite{KirRor02InfNonSimpleCalgAbsOInfty}) it was shown in \cite[Proposition~4.11]{KirRor02InfNonSimpleCalgAbsOInfty} that a \ca{} is weakly purely infinite if and only if its multiplier algebra is. In general, given a certain property $\mathcal{P}$, one can ask: If a \ca{} $A$ satisfies $\mathcal{P}$, when does $\mathcal{M}(A)$ satisfy $\mathcal{P}$?

In this paper we study the question above for the property of being \emph{nowhere scattered}. This notion ensures sufficient noncommutativity of the algebra, and can be characterized in a number of ways. As shown in \cite[Theorem~3.1]{ThiVil21arX:NowhereScattered}, a \ca{} $A$ is nowhere scattered if and only if it has no nonzero elementary ideal-quotients. This is in turn equivalent to no hereditary sub-\ca{} of $A$ admitting a one-dimensional representation. Every weakly purely infinite \ca{} is nowhere scattered (\cite[Example~3.3]{ThiVil21arX:NowhereScattered}) but, in contrast to the weakly purely infinite case, the multiplier algebra of a nowhere scattered \ca{} need not be nowhere scattered; see Examples \ref{exa:product} and \ref{exa:simple}.
 
 \begin{qst}\label{qst:NscaImpMANsca}
  Let $A$ be nowhere scattered. When is $\mathcal{M}(A)$ nowhere scattered?
 \end{qst}
 
 One of the motivations behind \autoref{qst:NscaImpMANsca} is the study of when a multiplier algebra has no characters, since such a property leads to important structure results. For example, it follows from \cite[Theorem~3.2]{MR3509135} that every element in a  nowhere scattered multiplier algebra can be written as the finite sum of commutators and products of two commutators. Moreover, knowing that $\mathcal{M}(A)$ is nowhere scattered would have implications on its unitary group; see \cite{ChandRobert}.
 
 Nowhere scatteredness of a \ca{} $A$ can also be characterized in terms of its Cuntz semigroup $\Cu (A)$, a powerful invariant for \ca{s} introduced in \cite{Cun78DimFct} and further developed in \cite{CowEllIva08CuInv}; see also \cite{AntPerThi18TensorProdCu},  \cite{AntPerRobThi18arX:CuntzSR1},  \cite{ThiVil22DimCu} and \cite{MR4310038}. Explicitly, it was shown in \cite[Theorem~8.9]{ThiVil21arX:NowhereScattered} that a \ca{} is nowhere scattered if and only if its Cuntz semigroup is \emph{weakly $(2,\omega )$-divisible}, a notion defined by Robert and R{\o}rdam in \cite{RobRor13Divisibility} to study when certain \ca{s} have characters; see \autoref{pgr:2OmegaDiv}. Consequently, the study of \autoref{qst:NscaImpMANsca} leads naturally to the study of divisibility properties of $\Cu (\mathcal{M}(A))$. The main divisibility properties at play in this case are weak $(2,\omega )$-divisibility and its bounded counterpart, known as \emph{weak $(m,n)$-divisibility}; see \autoref{sec:FinDiv}.

Although the \ca{s} from Examples \ref{exa:product} and \ref{exa:simple} are nowhere scattered (and thus have a $(2,\omega )$-divisible Cuntz semigroup), they both have \emph{unbounded divisibility}, that is, for every pair $m,n\in\NN$ there exists a Cuntz class that is not weakly $(m,n')$-divisible for any $n'\leq n$. Thus, we ask:
 
 \begin{qst}\label{qst:NSCaimpBound}
  Let $A$ be a nowhere scattered \ca{}. When does there exist $n\in\NN$ such that $[a]$ is weakly $(2,n)$-divisible for every $a\in A_+$?
  
  More generally, when can one find $m,n\in\NN$ such that $[a\otimes 1_m]$ is weakly $(2m,n)$-divisible for every $a\in A_+$?
 \end{qst}
 
 This question is studied in detail in \autoref{sec:BoundDiv}, where we provide a number of examples where one can answer \autoref{qst:NSCaimpBound} affirmatively. Further, we also show that having bounded divisibility is characterized by the soft part of the monoid; see \autoref{prp:BoundDivSoft}.
 
 We prove in \autoref{sec:NScaCorAlg} that \autoref{qst:NscaImpMANsca} has a positive answer whenever \autoref{qst:NSCaimpBound} does.
 
 \begin{thm}[\ref{thm:MainMA}]\label{thm:IntroMainMA}
 Let $A$ be a $\sigma$-unital \ca{}. Assume that for every orthogonal sequence $(a_i)_i$ of positive elements in $A$ there exist $m,n\in\NN$ such that $[a_i\otimes 1_m]$ is weakly $(2m, n)$-divisible for every $i$. Then $\mathcal{M}(A)$ is nowhere scattered.
 \end{thm}

 Using the study of bounded divisibility from \autoref{sec:BoundDiv}, and that nowhere scatteredness passes to ideals (\cite[Proposition 4.2]{ThiVil21arX:NowhereScattered}), one obtains:
 
 \begin{thm}[\ref{prp:FinNucDimMANSca}]
 Let $A$ be a $\sigma$-unital \ca{}. Assume that $A$ is of 
 \begin{itemize}
  \item[(i)] real rank zero, or
  \item[(ii)] finite nuclear dimension, or
  \item[(iii)] stable rank one with $k$-comparison.
 \end{itemize}
 
 Then $A$ is nowhere scattered if and only if $\mathcal{M}(A)$ is nowhere scattered.
 \end{thm}
 
 Further, one can generalize (iii) above by changing stable rank one for the condition of having a \emph{surjective rank map}; see \autoref{prp:SurjMap} and \autoref{rmk:Sr1SurjMap}.

 We finish the section by proving a weak converse of \autoref{thm:IntroMainMA} for stable \ca{s}:
 
 \begin{thm}[\ref{prp:MAStable}]
  Let $A$ be a $\sigma$-unital, stable \ca{}. Assume that $\mathcal{M}(A)$ is nowhere scattered. Then, for every $a\in A_+$ and $m\in\NN$ there exists $n$ such that $[a]$ is weakly $(m,n)$-divisible.
 \end{thm}
 
Most of the results in Sections \ref{sec:FinDiv} and \ref{sec:BoundDiv} can be translated directly to the more general setting of abstract Cuntz semigroups, or \CuSgp{s} for short; see for example \cite{AntPerThi18TensorProdCu} and \cite{GarPer22ModTheCu}. However, since in this paper we focus on multiplier algebras (which have no known $\Cu$-counterpart), we state all the results in the language of \ca{s} to ease the notation.

\subsection*{Acknowledgements.} The author thanks Hannes Thiel for his comments on a first draft of this paper.

\section{Preliminaries}

\begin{pgr}[Nowhere scatteredness and the Global Glimm Property]\label{pgr:GGP}
 As defined in \cite[Definition~A]{ThiVil21arX:NowhereScattered}, a \ca{} is \emph{nowhere scattered} if none of its quotients contains a minimal open projection. By \cite[Theorem~3.1]{ThiVil21arX:NowhereScattered}, a \ca{} $A$ is nowhere scattered if and only if no nonzero ideal-quotient of $A$ is elementary.
 
 One also says that a \ca{} $A$ has the \emph{Global Glimm Property} (\cite[Definition~4.12]{KirRor02InfNonSimpleCalgAbsOInfty}) if for every $a\in A_+$ and $\varepsilon >0$ there exists a \stHom{} $\varphi\colon M_2 (C_0 (0,1])\to \overline{aAa}$ such that the image of $\varphi$ contains $(a-\varepsilon )_+$. 
 
 A \ca{} is nowhere scattered whenever it has the Global Glimm Property. The converse remains open, and is known as the \emph{Global Glimm Problem}; see \cite{AntPerRobThi18arX:CuntzSR1}, \cite{EllRor06Perturb}, and \cite{ThiVil22arX:Glimm}.
 
 Examples of \ca{s} with the Global Glimm Property include simple, non-elementary \ca{s}, $\mathcal{Z}$-stable \ca{s}, and purely infinite \ca{s}. Traceless \ca{s} are all nowhere scattered.
\end{pgr}

\begin{pgr}[The Cuntz semigroup]
 For any pair of positive elements $a,b$ in a \ca{} $A$ one writes $a\precsim b$ if $a=\lim_n r_n b r_n^*$ for some sequence $(r_n)_n$ in $A$. One also says that $a$ is \emph{Cuntz equivalent} to $b$, in symbols $a\sim b$, whenever $a\precsim b$ and $b\precsim a$.
 
 The \emph{Cuntz semigroup} $\Cu (A)$ is defined as the quotient $(A\otimes \mathcal{K}_+)/\sim$ equipped with the order induced by $\precsim$ and the addition induced by diagonal addition; see \cite{Cun78DimFct} and \cite{CowEllIva08CuInv} for details.
 
 Given elements $x,y\in\Cu (A)$, we write $x\ll y$ whenever there exists $a\in (A\otimes \mathcal{K})_+$ and $\varepsilon >0$ such that $x\leq [(a-\varepsilon )_+]$ and $y=[a]$. As shown in \cite{CowEllIva08CuInv}, every increasing sequence in $\Cu (A)$ has a supremum, and every element can be written as the supremum of a $\ll$-increasing sequence.
 
 In recent years, the Cuntz semigroup has beneffited from the study of abstract Cuntz semigroups, or \emph{\CuSgp{s}} for short. Good references for this include \cite{AraPerTom11Cu}, \cite{AntPerThi18TensorProdCu}, and \cite{GarPer22ModTheCu}. In this paper, we will not use the language of \CuSgp{s} (except in \autoref{prp:Dim0BoundDiv}), but we will still make use of some of the abstract properties that the Cuntz semigroup of a \ca{} always satisfies. These are:
 \begin{itemize}
  \item[\axiomO{5}] Given $x',x,y',y,z\in\Cu (A)$ such that $x+y\leq z$ with $x'\ll x$ and $y'\ll y$, there exists $c\in\Cu (A)$ such that $y'\ll c$ and $x'+c\leq y\leq x+c$.
 \item[\axiomO{6}] Given $x'\ll x\ll y+z$, there exist $v\leq x,y$ and $w\leq x,z$ such that $x'\leq v+w$.
 \end{itemize}
 
 The reader is referred to \cite[Proposition~4.6]{AntPerThi18TensorProdCu} (also \cite{RorWin10ZRevisited}) and \cite{Rob13Cone} for the respective proofs.
\end{pgr}

\begin{pgr}[Divisibility in the Cuntz semigroup]\label{pgr:2OmegaDiv}
 Let $A$ be a \ca{}, let $x\in\Cu (A)$ and take $n\geq 1$ and $m\geq 2$. Following \cite{RobRor13Divisibility}, we will say that $x$ is \emph{weakly $(m,n)$-divisible} if, whenever $x'\ll x$, there exist elements $y_1,\ldots ,y_n$ such that $x'\leq y_1+\ldots +y_n$ and $my_j\leq x$ for each $j\leq n$. Similarly, $x$ is \emph{weakly $(m,\omega )$-divisible} if the previous condition holds without any bound on $n$, that is, allowing $n$ to depend on $x'$.
 
 The element $x$ is said to be \emph{$(m,n)$-divisible} (resp. $(m,\omega )$-divisible) if one can always set $y_1=\ldots =y_n$.
 
 A \ca{} is nowhere scattered if and only if every element in its Cuntz semigroup is weakly $(m,\omega )$-divisible for each $m$, whilst a \ca{} has the Global Glimm Property if and only if every element is $(m,\omega )$-divisible for each $m$; see \cite[Theorem~8.9]{ThiVil21arX:NowhereScattered} and \cite[Theorem~3.6]{ThiVil22arX:Glimm} respectively.
\end{pgr}

\section{Finite divisibility}\label{sec:FinDiv}

In this section we recall the notions of finite divisibility introduced in \cite{RobRor13Divisibility}; see \autoref{dfn:Div}. We prove their main properties (\autoref{prp:Weak2wSum}), and study some situations where $\Cu (A)$ contains a sup-dense subset of elements with finite divisibility. We also find sufficient conditions for finite weak divisibility to imply finite divisibility; see \autoref{prp:DiscGGP}.

\begin{dfn}[\cite{RobRor13Divisibility}]\label{dfn:Div}
 Let $a$ be a positive element of a \ca{} $A$, and let $m\geq 2$. We let $\wdiv_m ([a])$ and ${\rm Div}_m ([a])$ be the least positive integers $n,n'$ such that $[a]\in\Cu (A)$ is weakly $(m,n)$-divisible and $(m,n')$-divisible respectively. 
 
 If no such $n$ or $n'$ exist, we set $\wdiv_m ([a])=\infty $ or ${\rm Div}_m ([a])=\infty$.
\end{dfn}

\begin{rmk}\label{rmk:FinDivProj}
 The class of a positive element $a\in A_+$ in $\Cu (A)$ is said to be \emph{compact} if $[a]\ll [a]$. Every projection gives rise to a compact Cuntz class and, in some cases, these are the only compact classes; see \cite{BroCiu09IsoHilbModSF}. 
  
  Given a compact element $[a]\in\Cu (A)$, then $[a]$ is weakly $(m,\omega )$-divisible if and only if $\wdiv_m ([a])<\infty $.
  
  However, for a non-compact element $[a]$, $\wdiv_m ([a])<\infty$ is not equivalent to $[a]$ being weakly $(m,\omega )$-divisible; see \autoref{exa:DirSumInfDiv} and \autoref{rmk:ExaSimpNotDiv}.
\end{rmk}


Let us first summarize the main properties of $\wdiv_m ()$. Part (iii) of the following lemma is in analogy to \cite[Lemma~4.9]{KirRor02InfNonSimpleCalgAbsOInfty}, while (iv) is a $\Cu$-analogue of  \cite[Proposition~3.6]{RobRor13Divisibility}.

Recall that, given $x,y\in\Cu (A)$, the infimum $x\wedge \infty y$ always exists; see \cite[Remark~2.6]{AntPerRobThi21Edwards}. Here, $\infty y$ denotes the element $\sup_n ny$.

\begin{lma}\label{prp:Weak2wSum}
 Let $A$ be a \ca{}, and let $x=[a]\in\Cu (A)$ and $m\in\NN$. Then,
 \begin{itemize}
 \item[(i)]  $\wdiv_m (x)\leq \wdiv_{m'} (x)$ whenever $m\leq m'$.
 \item[(ii)] $\wdiv_m ([a])\leq N(\wdiv_m ([b_1]) +\ldots + \wdiv_m ([b_r]))$ whenever 
 \[
  [b_1],\ldots ,[b_r]\leq [a]\leq N([b_1]+\ldots +[b_r]).
 \]
  \item[(iii)] $\wdiv_m ([a+b])\leq \wdiv_m ([a]) + \wdiv_m ([b])$.
  \item[(iv)] Given $x,y\in\Cu (A)$ such that $y\leq x$ and $x\wedge \infty y=y$, then $\wdiv_m (y)\leq \wdiv_m (x)$.
  \item[(v)] $\wdiv_m (x)\leq \sup_k \wdiv_m (x_k)$ for every increasing sequence $(x_k)_k$ in $\Cu (A)$ with supremum $x$.
 \end{itemize}
\end{lma}
\begin{proof}
 (i) follows  directly from the definition of  $\wdiv_m ([a])$. 

 To see (ii), let $n_j=\wdiv_m ([b_j])$ for $j=1,\ldots ,r$. We may assume that all quantities are finite, since otherwise we are done.
 
 Take $x\ll [a]$. Since $[a]\leq N[b_1]+\ldots +N[b_r]$, we can find elements $x_j$ such that $x_j\ll [b_j]$ and $x\leq Nx_1+\ldots +Nx_r$. The element $[b_j]$ is weakly $(m,n_j)$-divisible, so we can find elements $y_{1,j},\ldots ,y_{n_j,j}$ such that 
 \[
  my_{i,j} \leq [b_j],\andSep 
  x_{j}\leq y_{1,j}+ \ldots +y_{n_j,j}
 \]
 for each $i$ and $j$.
 
 Using that $[b_j]\leq [a]$, one has $my_{i,j}\leq [a]$. Further, since we also have $x\leq N\sum_j x_j$, we deduce that $x\leq N\sum_{i,j} y_{i,j}$. It follows that $\sdiv_m ([a])\leq N(n_1+\ldots +n_r)$, as desired.
 
 Note that, given $a,b\in (A\otimes\mathcal{K})_+$, we have $[a+b]\leq [a]+[b]$ and $[a],[b]\leq [a+b]$. Thus, (iii) follows directly from (ii).
 
 For (iv), let $x,y\in \Cu (A)$ be as stated. As before, we may assume that $\wdiv_m (x)=n$ for some $n\in\NN$. Take $y'\ll y$, and apply the weak $(m,n)$-divisibility of $x$ to obtain elements $y_1,\ldots ,y_n$ satisfying the properties from \autoref{pgr:2OmegaDiv} for $y'$ and $x$.
 
 Since $my_j\leq x$ for every $j$, and $x\wedge \infty y=y$, one obtains $m(y_j\wedge \infty y)\leq x\wedge \infty y=y$. It is now easy to check that the elements $y_j\wedge \infty y$ satisfy the conditions from \autoref{pgr:2OmegaDiv} for $y'$ and $y$, as required.
 
 To prove (v), assume that the supremum $\sup_k \wdiv_m (x_k)$ is finite, since otherwise there is nothing to prove. Let $n\in\NN$ be such that $\sup_k \wdiv_m (x_k)\leq n$, and take $x'\in \Cu (A)$ such that $x'\ll x$. Since $x=\sup_k x_k$, it follows that $x'\ll x_k$ for some $k\in\NN$. Thus, using that $\wdiv_m (x_k)\leq n$, we obtain elements $y_1,\ldots ,y_n$ with $x'\leq y_1+\ldots +y_n$ and $my_j\leq x_k\leq x$ for each $j$. We get that $\wdiv_m (x)\leq n$, as desired.
\end{proof}


Recall that a \ca{} $A$ is said to have \emph{strict comparison} if $[a]\leq [b]$ in $\Cu (A)$ whenever there exists $\delta>0$ such that $\lim_n \tau (a^{1/n})\leq (1-\delta)\lim_n \tau (b^{1/n})$ for every $2$-quasitrace $\tau$; see \cite{EllRobSan11Cone} for more details.

\begin{lma}\label{prp:DivStComp}
 Let $A$ be a \ca{}, let $[a]\in \Cu (A)$ and $m\in\NN$. Then,
 \[
  \wdiv_{km}(k[a])\leq \wdiv_m ([a])
 \]
 for every $k\in\NN$. 
 
 If $A$ has strict comparison, one has $\wdiv_{m} ([a])\leq \wdiv_{k(m+1)}(k[a])$.
\end{lma}
\begin{proof}
 Let $n=\wdiv_m ([a])$, which we may assume to be finite, and take $x\ll k[a]$. Then, there exists $\varepsilon >0$ such that $x\ll k[(a-\varepsilon )_+]$. Since $[(a-\varepsilon )_+]\ll [a]$, there exist $y_1,\ldots ,y_n\in\Cu (A)$ such that 
 \[
  [(a-\varepsilon )_+]\leq y_1+\ldots +y_n,\andSep 
  my_j\leq [a]
 \]
 for each $j$.
 
 In particular, one gets
 \[
  x\ll k[(a-\varepsilon )_+]\leq ky_1+\ldots +ky_n,\andSep 
  m(ky_j)\leq k[a],
 \]
 which implies $\wdiv_{km}(k[a])\leq n$, as desired.
 
 Now assume that $A$ has strict comparison of positive elements, and let $n=\wdiv_{k(m+1)}(k[a])$. As before, we may assume $n$ to be finite. Let $x\ll [a]$, which implies $kx\ll k[a]$. Then, there exist $z_1,\ldots ,z_n$ such that $x\leq kx\leq z_1+\ldots +z_n$ and $k(m+1)z_j\leq k[a]$. Thus, one has $(km+1)mz_j\leq km(m+1)z_j\leq (km)[a]$.
 
 Using \cite[Proposition~6.2]{EllRobSan11Cone}, we get $mz_j\leq [a]$. This implies $\wdiv_{m} ([a])\leq \wdiv_{k(m+1)}(k[a])$.
\end{proof}

\begin{rmk}
 Let $A$ be a \ca{} and let $m\in\NN$. Assume that $\wdiv_{m} ([a])<\infty$ for every $a\in A_+$. \autoref{prp:Weak2wSum}~(ii) implies, in particular, that $\wdiv_m ( [b] )<\infty$ for every $b\in M_n (A)_+$ and $n\in\NN$. 
 
 Indeed, given $b\in M_n(A)_+$ there exist elements $b_1,\ldots , b_n\in A_+$ and $N\in\NN$ such that $[b_i]\leq [b]\leq N([b_1]+\ldots +[b_n])$ for each $i\leq n$; see \cite[4.2]{AntPerThi20CuntzUltraproducts} and \cite[Lemma~3.3]{ThiVil22arX:Glimm}.
 
 By \autoref{prp:Weak2wSum} (ii), we have that $\wdiv_{m} (b)<\infty$, as desired.
\end{rmk}

\begin{exa}\label{exa:DirSumInfDiv}
 Let $(A_k)_k$ be a family of unital \ca{s} such that 
 \[
  k\leq \wdiv_2 ([1_{A_k}])\leq 7k
 \]
 for each $k$. 
 
 By \cite[Theorem~7.9]{RobRor13Divisibility}, such \ca{s} can be taken to be simple, unital, infinite dimensional AH-algebras. Using \cite[Proposition~4.13]{ThiVil21arX:NowhereScattered}, one sees that $A:=\oplus_k A_k$ is nowhere scattered. Thus, every element in $\Cu (A)$ is weakly $(2,\omega )$-divisible.
 
 Consider the element $x=\sup_n \sum_{k=1}^n [1_{A_k}]\in \Cu (\oplus_k A_k)$, and denote by $\iota_k$ the induced inclusion $\Cu (A_k)\to \Cu (\oplus_k A_k)$. Then, 
 \[
  \iota_k ([1_{A_k}])\leq x,\andSep 
  x\wedge \infty \iota_k ([1_{A_k}]) = \iota_k ([1_{A_k}])
  .
 \]

 Thus, \autoref{prp:Weak2wSum}~(iv) implies that 
 \[
  k\leq \wdiv_2 (\iota_k ([1_{A_k}]))\leq \wdiv_2 (x).
 \]
 
 This shows that $x$ is weakly $(2,\omega )$-divisible, but $\wdiv_2 (x)=\infty$; see also \autoref{prp:CharacCharac}.
\end{exa}
 
We now move our attention to \ca{s} whose Cuntz semigroup $\Cu (A)$ satisfies that every element (or, at least, a finite multiple of it) has finite weak divisibility.

To study such \ca{s}, we will need the following definitions:

\begin{dfn}
 Let $A$ be a \ca{}, and let $a\in A_+$. We set 
 \begin{align*}
  {\rm mDiv} ([a]) &= \inf \{ \sdiv_{2m} (m[a]) \mid m\in\NN \} \\
  \mwdiv ([a]) &= \inf \{ \wdiv_{2m} (m[a]) \mid m\in\NN \}
 \end{align*}
\end{dfn}
 
 We will denote by  $\Cu (A)_{\wdiv}$ (respectively $\Cu (A)_{\mwdiv}$ and $\Cu (A)_{\rm mDiv}$) the subset of $\Cu (A)$ consisting of the elements $[a]\in \Cu (A)$ such that $\wdiv_{2m} ([a])<\infty$ for some $m\geq 1$ (respectively $\mwdiv ([a])<\infty$ and ${\rm mDiv}([a])< \infty$). By \autoref{prp:Weak2wSum}, the subsets $\Cu (A)_{\wdiv}$ and $\Cu (A)_{\mwdiv}$ are always subsemigroups of $\Cu (A)$.
 
 Note that one has $\Cu (A)_{\wdiv}\subseteq \Cu (A)_{\mwdiv}$, and it is easy to see that $\Cu (A)_{\wdiv}= \Cu (A)_{\mwdiv}$ whenever $\Cu (A)$ is unperforated.
 
 \begin{exa}\label{exaDivCC}
 If $A= \CC$, its Cuntz semigroup is isomorphic to $\NNbar:=\NN\cup\{\infty \}$. It is readily checked that $\wdiv_m (x)<\infty$ if and only if $x\geq m$. Thus, one has $\Cu (A)_{\wdiv} = \Cu (A)_{\mwdiv} = \{ 0,2,3, \ldots , \infty\}$.
 \end{exa}
 
  \begin{exa}\label{exa:exaDivZ}
   Every tracially $\mathcal{Z}$-stable \ca{} $A$ satisfies $\Cu (A)=\Cu (A)_{\wdiv}$. More generally, recall that a \ca{} is said to be \emph{$N$-almost divisible} (see \cite{Win12NuclDimZstable}) if for each pair $x',x\in\Cu (A)$ satisfying $x'\ll x$, and $k\in\NN$, there exists $y\in\Cu (A)$ such that 
   \[
    ky\leq x,\andSep 
    x'\leq (k+1)(N+1)y.
   \]

   Thus, if $A$ is $N$-almost divisible for some $N\in\NN$, one also has $\Cu (A)=\Cu (A)_{\wdiv }$.
   
   By \cite[Theorem~3.1]{RobTik17NucDimNonSimple}, every nowhere scattered \ca{} of nuclear dimension $N$ that has no nonzero, simple, purely infinite ideal-quotients is $N$-almost divisible. This includes residually stably finite, nowhere scattered \ca{s} of finite nuclear dimension, and nowhere scattered \ca{s} of finite decomposition rank. Consequently, one has $\Cu (A)_{\wdiv}=\Cu (A)$.

   As we will see in \autoref{prp:FNucDimImpMWDiv}, $\Cu (A)_{\mwdiv}=\Cu (A)$ for any nowhere scattered \ca{} of finite nuclear dimension. In fact, one has more: $\Cu (\oplus_{i=1}^\infty A)_{\rm mDiv} = \Cu (\oplus_{i=1}^\infty A)$.
  \end{exa}

  \begin{exa}\label{exa:exaDivRR0}
   As shown in \cite[Theorem~9.1]{ThiVil21arX:NowhereScattered}, every element $x$ in the Cuntz semigroup of a nowhere scattered, real rank zero \ca{} is \emph{weakly divisible}, that is, there exist $y,z$ such that 
   $x=2y+3z$. In particular, one gets 
   \[
    2(y+z)\leq x\leq 3(y+z).
   \]
  
  Thus, it follows that $\wdiv_2(x)\leq 3$ for every $x$ and, consequently, that $\Cu (A)_{\wdiv} = \Cu (A)$.
 \end{exa}
 
 As noted in \autoref{exa:exaDivRR0} above, every element in the Cuntz semigroup of a nowhere scattered, real rank zero \ca{} has finite divisibility. Lemmas \ref{prp:SimpFinDiv} and \ref{prp:TopDimZeroDense} below show that the same is true for a sup-dense subset of (non-elementary) simple \ca{s} and nowhere scattered \ca{s} of topological dimension zero. However, this may not imply that every element in such a \ca{} has finite divisibility. 
 
 Recall that an element $x\in\Cu (A)$ is \emph{idempotent} if $x=2x$. In the Cuntz semigroup of a simple \ca{} there is only one nonzero, idempotent element, which we denote by $\infty$.
 
 \begin{lma}\label{prp:SimpFinDiv}
  Let $A$ be a simple, non-elementary \ca{}. Then, $\Cu (A)_{\wdiv}$ is sup-dense in $\Cu (A)$.
  
  If $A$ is not stably finite, then $\Cu (A)_{\wdiv}=\Cu (A)$.
 \end{lma}
 \begin{proof}
  Let $m\in\NN$. If $A$ is simple and non-elementary, it follows from \cite[Proposition~5.2.1]{Rob13Cone} that $\Cu (A)$ is $(m,\omega )$-divisible. 
  
  Let $x\in\Cu (A)$ be such that $x\ll \infty$. Then, using that $x$ is $(m,\omega )$-divisible, there exists a nonzero element $y\in\Cu (A)$ such that $my\leq x$.  Using that $x\ll \infty = \infty y$, we find $n\in\NN$ such that $x\leq ny$. This implies that $\wdiv_m (x)\leq n$. Since the subset of elements compactly contained in $\infty$ is sup-dense in $\Cu (A)$, the result follows.
  
  Now assume that $A$ is not stably finite. Then, the element $\infty$ in $\Cu (A)$ is compact; see, for example, \cite[Lemma~6.21]{Thi17:CuLectureNotes} or the proof of \cite[Theorem~2.6]{Bos22StableElements}.
  
  It follows that every element in $\Cu (A)$ is compactly contained in $\infty$. Using the argument above, we see that $\Cu (A)_{\wdiv}=\Cu (A)$.
 \end{proof}
 
 \begin{rmk}
  An inspection of the proof shows that \autoref{prp:SimpFinDiv} is slightly more general: For a \ca{} $A$, one has $\Cu (A)_{\wdiv}=\Cu (A)$ whenever every idem-multiple element in $\Cu (A)$ is compact.
 \end{rmk}
 
 As defined in \cite[Remark~2.5~(vi)]{BroPed09Limits}, a \ca{} has \emph{topological dimension zero} whenever its primitive ideal space has a basis of compact-open subsets. Examples include all \ca{s} with the ideal property.
 
 \begin{lma}\label{prp:TopDimZeroDense}
  Let $A$ be a separable nowhere scattered \ca{} of topological dimension zero. Then $\Cu (A)_{\wdiv}$ is dense in $\Cu (A)$.
 \end{lma}
 \begin{proof}
  Take $x\in \Cu (A)$, and assume that there exists $x'\in \Cu (A)$ such that $x'\ll x\ll \infty x'$.
  
  Let $M\in\NN$ be such that $x\leq Mx'$. Using that $x$ is weakly $(m,\omega )$-divisible, one finds $y_1,\ldots , y_n$ such that $my_i\leq x$ and $x'\leq y_1+\ldots y_n$. Thus, one has $x\leq My_1+\ldots +My_n$, and we get $\wdiv_m (x)\leq Mn$.
  
  It follows from \cite[Proposition~4.18]{ThiVil22arX:Glimm} that such elements $x$ are sup-dense in $\Cu (A)$, as required.
 \end{proof}
 
 \subsection{The Global Glimm Problem}
 We finish this section by studying when $\Cu (A)_{\rm div}=\Cu (A)$ implies $\Cu (A)_{\rm Div}=\Cu (A)$. More concretely, we study when an element of finite weak divisibility has finite divisibility. Given its similarities with the Global Glimm Problem (see \autoref{pgr:GGP}), one could call this the \emph{discrete Global Glimm Problem}.
 
 Following the ideas from \cite[Section~6]{ThiVil22arX:Glimm}, we obtain:
 
 \begin{thm}\label{prp:DiscGGP}
  Let $A$ be a \ca{} satisfying $\Cu (A)_{\rm div}=\Cu (A)$. Assume that there exist $k\in\NN$ and maps $N,M\colon\NN\to\NN$ such that 
  \begin{itemize}
   \item[(i)] whenever $x'\ll x\leq ny,nz$, there exists $t\in\Cu (A)$ such that $x'\ll N(n)t$ and $t\ll y,z$;
   \item[(ii)] whenever $x'\ll x$ and $2x\ll y+2nz$, there exists $g\in\Cu (A)$ such that $2g\ll y$ and $x'\ll g+M(n)z$;
   \item[(iii)] whenever $x_1+f,x_2+f\ll y$ and $x_i'\ll x_i\ll f$ for $i=1,2$, one can find $z_1,z_2\in\Cu (A)$ such that $z_1+z_2\ll y$ and $x_1'+x_2'\ll kz_1 ,kz_2$.
  \end{itemize}
  
  Then, $\Cu (A)_{\sdiv}=\Cu (A)$.
 \end{thm}
 \begin{proof}
  Assume that (i)-(iii) are satisfied, and let $x\in \Cu (A)$. By assumption, we have $\wdiv_{m}(x)<\infty $, which implies $\wdiv_2 (x)<\infty$. Set $n:= \wdiv_{2}(x)$. We will show that ${\rm Div}_2(x)<\infty$ by proving that 
  \[
   {\rm Div}_2(x)\leq \Big(N\big(\max\{ N (2N_{2,n}) ,M(n)\}\big)
   +N\big(N (2N_{2,n})\big)\Big)N(k),
  \]
  where $N_{2,n}:= (N\circ\stackrel{n}{\ldots}\circ N) (2)$.
  
  Let $x'\in\Cu (A)$ be such that $x'\ll x$, and take $y' , y$ such that $x'\ll y'\ll  y\ll x$. Then, since $x$ is $(2,n)$-divisible, we obtain $y_1,\ldots ,y_n$ satisfying $y\ll \sum y_j$ and $2y_j\leq x$ for each $j$. Take $y_j'\ll y_j$ such that $y\ll \sum y_j'$. By \cite[Lemma~2.2]{ThiVil22arX:Glimm}, for each $j$ we obtain an element $r_j$ such that $y_j+r_j\leq x\leq 2r_j$ and $y_j'\leq r_j$.

  Using (i), we find $r\in\Cu (A)$ satisfying $y \ll N_{2,n} r$ and $r\leq r_j$ for every $j$. Take $r'\ll r$ such that $y\ll N_{2,n} r'$. Using \axiomO{5} at $r'\ll r\leq x$, we obtain $c\in\Cu (A)$ such that $r'+c\leq x\leq r+c$.
   
   In particular, since one has $y_j +r\leq x\leq r+c$, we obtain
   \[
   \begin{split}
    2y &\leq 2y_1+\ldots +2y_n\leq (2y_1+\ldots +2y_n)+r \\
    &\leq (y_1+2y_2+\ldots +2y_n) + r+ c\leq \ldots \leq r+(2n)c.
   \end{split}
   \]
 
   Using (ii), we find elements $g',g\in\Cu (A)$ such that 
   \[
    2g\ll r,\quad 
    y'\ll g'+M(n) c,\andSep 
    g'\ll g.
   \]
 
   Applying \cite[Lemma~2.2]{ThiVil22arX:Glimm} at $2g\ll r$, we find $d\in\Cu (A)$ such that $g'+d\leq r\leq 2d$. Thus, we have $y'\ll y \leq N_{2,n} r'\leq 2N_{2,n} d$. By (i), we find $f\in\Cu (A)$ such that 
   \[
    y' \leq N(2N_{2,n}) f,\andSep 
    f\ll r',d.
   \]
   
   By \axiomO{6} applied at $x'\ll y'\ll g'+M(n) c$, one finds elements $s',s,t',t\in\Cu (A)$ such that 
   \[
    x'\ll s'+t',\quad 
    s'\ll s\leq y', g',\andSep 
    t'\ll t\leq y', M(n) c.
   \]
 
   Set $N_1:=N(\max\{ N (2N_{2,n}) ,M(n)\})$ and $N_2:=N(N (2N_{2,n}))$. Applying (i) at $t' \ll t\leq M(n) c$ and $t'\ll t\leq y'\leq N (2N_{2,n}) f$, we obtain $x_1\in\Cu (A)$ such that $x_1\ll c,f$ and $t'\ll N_1 x_1$. Take $x_1',x_1''\in\Cu (A)$ such that $x_1'\ll x_1''\ll x_1$ and $t'\leq N_1x_1'$.
   
   Using (i) again, but this time at $s'\ll s\leq y'\leq N (2N_{2,n}) f$, we obtain $x_2\in\Cu (A)$ such that $x_2\ll f,s$ and $s'\ll N_2 x_2$. As before, take $x_2',x_2''\in\Cu (A)$ such that $x_2'\ll x_2''\ll x_2$ and $s'\leq N_2x_2'$.
   
   One has 
   \[
    x_1 + f\ll c+ r'\leq x,\quad 
    x_2+f \ll s+d\leq g'+d\leq r\leq x,\andSep x_1,x_2\ll f.
   \]
   
   By (iii), we find $z_1 ,z_2\in\Cu (A)$ such that 
   \[
    z_1+z_2\leq x,\andSep 
    x_1''+x_2''\ll kz_1 ,kz_2.
   \]
   
   Applying (i) one last time, we find $z\in\Cu (A)$ such that 
   \[
    z\leq z_1,z_2 ,\andSep x_1'+x_2'\leq N(k) z.
   \]
   
   The element $z$ satisfies $2z\leq z_1+z_2\leq x$, and 
   \[
    x'\leq t'+s'\leq (N_1+N_2)(x_1'+x_2')\leq (N_1+N_2)N(k) z
    ,
   \]
   as desired.
 \end{proof}
 
 One can check that real rank zero, or stable rank one \ca{s} satisfy (i)-(iii) in \autoref{prp:DiscGGP} above. In particular, this recovers the discrete part of \cite[Theorem~5.5]{AntPerRobThi18arX:CuntzSR1}. 
 
 \begin{cor}
  Let $A$ be a \ca{}. Assume that $A$ is either of stable rank one, or real rank zero. Then, $\Cu (A)_{\rm div}=\Cu (A)$ if and only if $\Cu (A)_{\rm Div}=\Cu (A)$.
 \end{cor}
 \begin{proof}
  If $A$ has real rank zero, use \autoref{exa:exaDivRR0}. For stable rank one \ca{s}, this just amounts to an inspection of a number of proofs:
  
  \cite[Lemma~5.3]{AntPerRobThi18arX:CuntzSR1} implies that (i) is  satisfied with $N=2n-1$.
  
  The proof of \cite[Lemma 5.5]{ThiVil22arX:Glimm}, but using weak cancellation instead of residually stably finiteness, gives (ii) in \autoref{prp:DiscGGP} with $k=2$.
  
  Finally, (iii) follows from an inspection of \cite[7.6-7.8]{ThiVil21arX:NowhereScattered} using that, when $A$ has stable rank one,  \axiomO{8} can be used without the assumption of an element being idempotent; see Definition~7.2 and Proposition~7.5 in \cite{ThiVil21arX:NowhereScattered} for more details.
 \end{proof}

 \begin{rmk}
 One can actually show that $\Cu (A)_{\rm mdiv}=\Cu (A)$ if and only if $\Cu (A)_{\rm mDiv}=\Cu (A)$ whenever $A$ is of stable rank one. Indeed, given $x\in\Cu (A)$ such that $\wdiv_{2m}(mx)<\infty$, it follows directly from \cite[Theorem~5.5]{AntPerRobThi18arX:CuntzSR1} that $\sdiv_{2m}(mx)<\infty$.
 
 The same result holds for real rank zero \ca{s} by \autoref{exa:exaDivRR0}.
 \end{rmk}
 
 \begin{qst}
  Let $A$ be a \ca{} such that $\Cu (A)=\Cu_{\rm mdiv} (A)$. When does $A$ satisfy $\Cu (A)=\Cu_{\rm mDiv} (A)$?
 \end{qst}
 
 \section{Bounded divisibility}\label{sec:BoundDiv}
 
 We focus in this section on \ca{s} where there is a global bound for the divisibility of the elements. That is, we look at those \ca{s} $A$ such that there exist $n,m\in\NN$ with $\wdiv_m ([a])\leq n$ for every $a\in A_+$. More generally, we focus on those  that satisfy $\sup_{a\in A_+}\wdiv_{2m} (m[a])<\infty$ for some  $m\in\NN$. As we will see in \autoref{thm:MainMA}, these algebras will have a nowhere scattered multiplier algebra.
 
 We begin the section by providing sufficient conditions for this bounded divisibility to occur (Propositions \ref{prp:FNucDimImpMWDiv} and \ref{prp:SurjMap}), and by showing that \ca{s} with this property are closed under extensions (\autoref{prp:BoundDivExt}). Note that we have already seen some examples, such as tracially $\mathcal{Z}$-stable \ca{s}, or \ca{s} of real rank zero; see Examples \ref{exa:exaDivZ} and \ref{exa:exaDivRR0}.
 
 In \autoref{prp:BoundDivSoft}, we prove that the set of soft elements characterizes when a \ca{} has multiples of bounded divisibility, and we use this to deduce the property for simple, weakly cancellative \ca{s} of Cuntz covering dimension zero; see \autoref{prp:Dim0BoundDiv}.\vspace{0.2cm}
 
 As mentioned in \cite{RobTik17NucDimNonSimple}, it is unclear if $\dim_{\rm nuc} (A)<\infty$ (with $A$ nowhere scattered) implies $\wdiv (a)<\infty$ for every $a\in A_+$. However, one stills gets the following result, which is a direct consequence of \cite[Proposition~3.2]{RobTik17NucDimNonSimple}.
 
 \begin{prp}\label{prp:FNucDimImpMWDiv}
  Let $A$ be a nowhere scattered \ca{} of finite nuclear dimension at most  $k$. Then, for every $[a]\in \Cu (A)$, one has 
  \[
   \wdiv_{4(k+1)} (2(k+1)[a])\leq 8(k+1)^2.
  \]
  
  In particular, in a nowhere scattered \ca{} of finite nuclear dimension, one always gets $\mwdiv ([a])<\infty$.
 \end{prp}
 \begin{proof}
  If $A$ is nowhere scattered and of nuclear dimension at most $k$, so is $A\otimes \mathcal{K}$.
 
  Let $[a]\in \Cu (A)$. Upon passing to the hereditary \ca{} $\overline{a(A\otimes \mathcal{K})a}$, we may assume that $a$ is full in $A\otimes \mathcal{K}$; note that $\overline{a(A\otimes \mathcal{K})a}$ is still nowhere scattered by \cite[Proposition~4.1]{ThiVil21arX:NowhereScattered}, and has nuclear dimension at most $k$ by \cite[Proposition~2.5]{WinZac10NuclDim}. 
  
  Recall from \cite[Theorem~3.1]{ThiVil21arX:NowhereScattered} that $A\otimes\mathcal{K}$ is nowhere scattered if and only if it has no finite-dimensional irreducible representations. By \cite[Proposition~3.2]{RobTik17NucDimNonSimple}, there exists $[b]\in \Cu (A)$ such that 
  \[
   [(a-\varepsilon )_+]\leq 4(k+1)[b],\andSep 
   4(k+1)[b]\leq 2(k+1)[a].
  \]

  Thus, we obtain that $2(k+1)[(a-\varepsilon )_+]\leq 8(k+1)^2[b]$. Since $2(k+1)[a]$ can be written as the supremum of the elements $2(k+1)[(a-\varepsilon )_+]$ with $\varepsilon\to 0$, it follows that $\wdiv_{4(k+1)} (2(k+1)[a])< 8(k+1)^2$, as desired.
 \end{proof}
 
  Let us denote by $F(A)$ the central sequence algebra, as defined in \cite[Definition~1.1]{Kir06CentralSeqPI}. Given a separable, unital \ca{} $A$, it is not known if $A$ being $\mathcal{Z}$-stable is equivalent to $F(A)$ admitting no characters; see \cite{KirRor15CentralSeqCharacters}. Although this question remains open, one does have the following result from \cite[Part~II,~Article~B,~Proposition~2.8]{Chris17PhD}.
 
 \begin{prp}[\cite{Chris17PhD}]
  Let $A$ be a separable \ca{}. Assume that $F(A)$ admits no characters. Then, for every $m\in\NN$ there exists $n$ such that $\wdiv_m ([a])\leq n$ for all $a\in A_+$.
 \end{prp}
 \begin{proof}
  It follows from \cite[Corollary~5.6]{RobRor13Divisibility} that $F(A)$ has no characters if and only if $\wdiv_2 ([1])<\infty $ in $\Cu (F(A))$. Let $n_0=\wdiv_2 ([1])$. \cite[Part~II,~B,~Proposition~2.8]{Chris17PhD} implies that $\wdiv_2 ([a])\leq n_0 $ for every $a\in A_+$.
  
  An inductive argument now shows that $\sup_{a\in A_+}\wdiv_m ([a])<\infty$ for every $m\in\NN$; see, for example, \cite[Lemma~3.4]{ThiVil22arX:Glimm}.
 \end{proof}
 
 Given $[a],[b]\in\Cu (A)$, recall that we write $[a]<_s [b]$  whenever there exists $\gamma <1$ such that $\lambda ([a])\leq \gamma \lambda ([b])$ for every $\lambda \in F(\Cu (A))\cong {\rm QT}(A)$; see \cite[2.1]{RobTik17NucDimNonSimple}.

Let $k\in\NN$. As defined in \cite[Definition~2.1]{Win12NuclDimZstable}, a \ca{} $A$ is said to have \emph{$k$-comparison} if for every $[a],[b_0],\ldots ,[b_k]\in\Cu (A)$ such that $[a]<_s [b_i]$ for all $i$, we have $[a]\leq \sum_i [b_i]$.

Also recall that $A$ has a \emph{surjective rank map} if every element in $L (F(\Cu (A)))$ can be realized as a rank function; see \cite[Section~7]{AntPerRobThi18arX:CuntzSR1} for details.
 
 \begin{prp}\label{prp:SurjMap}
 Let $A$ be a nowhere scattered \ca{} with $k$-comparison. Assume that the family of separable sub-\ca{s} of $A$ with a surjective rank map is $\sigma$-complete and cofinal. Then, $\sdiv_{2(k+1)}((k+1)[a])\leq 4(k+1)^3$ for every $[a]\in\Cu (A)$.
 \end{prp}
 \begin{proof}
  Let $x:=[a]$. By \cite[Proposition~6.1]{ThiVil21DimCu2} and \cite[Proposition~3.8.1]{FarHarLupRobTikVigWin21ModelThy}, there exists a separable sub-\ca{} $B\subseteq A$ that has a surjective rank map, is nowhere scattered, contains $x$, and is such that the induced inclusion map $\Cu (B)\to \Cu (A)$ is an order-embedding.
  
  By definition, every element in $L (F(\Cu (B)))$ can be realized as a rank function. Thus, there exists $y\in\Cu (B)$ such that $\widehat{y}=\frac{1}{3(k+1)}\widehat{x}$. This implies, in particular, that
  \[
   2(k+1)\widehat{y}\leq \frac{2}{3}\widehat{x},\andSep 
   \widehat{x}\leq 3(k+1)\widehat{y}= \frac{3}{4}(4(k+1))\widehat{y}.
  \]

  Applying $k$-comparison, one obtains $2(k+1)y\leq (k+1)x$ and $x\leq 4(k+1)^2 y$ in $\Cu (B)$. Using that the map $\Cu (B)\to \Cu (A)$ is an order-embedding, the same is true in $\Cu (A)$. Thus, we have 
  \[
   2(k+1) y\leq (k+1)x\leq 4(k+1)^3 y,
  \]
  as desired.
 \end{proof}
 
 \begin{rmk}\label{rmk:Sr1SurjMap}
  As shown in \cite[Theorem~7.14]{AntPerRobThi18arX:CuntzSR1}, a separable, nowhere scattered \ca{} of stable rank one always has a surjective rank map. Thus, \autoref{prp:SurjMap} above applies to all stable rank one \ca{s} with $k$-comparison.
  
  In fact, it was shown in \cite[Theorem~8.12]{AntPerRobThi18arX:CuntzSR1} that a separable, nowhere scattered \ca{} of stable rank one has $k$-comparison if and only if it has strict comparison. Consequently, we get that $\wdiv_2 (x)\leq 4$ for every $x\in\Cu (A)$. Further, a small modification of the proof shows that $\Cu (A)$ is almost divisible (ie. $0$-almost divisible).
 \end{rmk}
 
The following is a generalization of \cite[Proposition~4.2]{ThiVil21arX:NowhereScattered}.
 
 \begin{prp}\label{prp:BoundDivExt}
  Let $I$ be an ideal of a \ca{} $A$. Assume that there exist $m_0,m_1,n_1,n_2\in\NN$ such that 
  \[
   \wdiv_{2m_0}(m_0[a])\leq n_0,\andSep 
   \wdiv_{2m_1}(m_1 [b])\leq n_1
  \]
  for every $a\in (A/I)_+$ and every $b\in I_+$. Then, 
  \[
   \wdiv_{2m_0 m_1}((m_0 m_1 )[c])\leq m_0 (m_1 n_0 +n_1 )
  \]
  for every $c\in A_+$.
  
  Conversely, $\wdiv_{2m}(m[a]),\wdiv_{2m}(m[b])\leq n$ for every $a\in (A/I)_+$ and $b\in I_+$ whenever $\sup_{c\in A_+}\wdiv_{2m}(m[c])\leq n$ for some $m,n\in\NN$.
 \end{prp}
 \begin{proof}
  It follows from \cite{CiuRobSan10CuIdealsQuot} that $\Cu (I)$ can be identified with an ideal of $\Cu (A)$, and that $\Cu (A/I)$ is naturally isomorphic to the quotient $\Cu (A)/\Cu (I)$. We will make these identifications without explicitly writing the isomorphisms.
 
  Let $c\in A_+$, and set $x:=[c]$. Denote by $\Cu (\pi )\colon \Cu (A)\to \Cu (A/I)$ the induced quotient map. Given $x',x''$ in $\Cu (A)$ such that $x'\ll x''\ll x$, we can apply the bounded divisibility of $\Cu (A/I)$ to obtain elements $y_j'',y_j',y_j\in \Cu (A)$ for $j=1,\ldots ,n_0$ such that 
  \[
   (2m_0 )\Cu (\pi ) (y_j)\leq \Cu (\pi ) (m_0 x),\quad 
   \Cu (\pi )(m_0 x'')\leq \Cu (\pi ) (y_1'')+\ldots +\Cu (\pi ) (y_{n_0 }''),
  \]
  and $y_j''\ll y_j '\ll y_j$ for each $j$. Note that this is possible because $\Cu (\pi ) (x)=[\pi (c)]$ and $\Cu (\pi ) (x)$ has a representative in $(A/I)_+$.
  
  Take $w\in \Cu (I)$ such that $w=2w$, and
  \[
   (2m_0 )y_j'\ll (2m_0 )y_j\leq m_0 x+w,\andSep 
   m_0 x''\leq y_1''+\ldots + y_{n_0 }'' +w.
  \]
  
  Using \cite[Proposition~7.8]{ThiVil21arX:NowhereScattered}, we  find elements $z_j', z_j\in \Cu (A)$ satisfying 
  \[
   (2m_0 )z_j\ll m_0 x,\quad 
   y_j''\ll z_j'+w,\andSep 
   z_j'\ll z_j
  \]
  for every $j=1,\ldots ,n_0$.
  
  In particular, we have 
  \[
   x'\ll x''\leq m_0 x''\leq z_1'+\ldots + z_{n_0 }' +w.
  \]

  Applying \axiomO{6}, we obtain an element $r\in \Cu (A)$ such that 
  \[
   x'\ll z_1'+\ldots + z_{n_0 }' +r,\andSep 
   r\leq x'',w.
  \]

  Note that, since $r\leq x''$, it follows that $r=[b]$ for some $b\in A_+$. Further, since we also have $[b]\leq w$, one gets $b\in I_+$. Take $r'\ll r$ such that $x'\leq z_1'+\ldots + z_{n_0}' +r'$.
  
  Now, using that $\wdiv_{2m_1}(m_1[b])\leq n_1$, one gets elements $r_1,\ldots , r_{n_1}$ such that 
  \[
   (2m_1)r_i\ll m_1 r,\andSep 
   m_1 r'\ll r_1+\ldots +r_{n_1}
  \]
  for each $i\leq n_1$.

  Thus, it follows that 
  \[
   (m_0 m_1 )x'\leq (m_0 m_1 )z_1'+\ldots +(m_0 m_1)z_{n_0}'+m_0 r_1+\ldots +m_0 r_{n_1}
  \]
  and 
  \[
   (2m_0 m_1)z_j'\leq (m_0 m_1)x,\quad 
   (2m_0 m_1)r_i\leq (m_0 m_1)r\leq (m_0 m_1)x
  \]
  for each $i$ and $j$.
  
  This proves that 
  \[
   \wdiv_{2m_0 m_1}((m_0 m_1)x)\leq m_0(m_1 n_0+n_1),
  \]
  as desired.
  
  The converse follows from a standard argument; see \cite[Proposition~3.9]{ThiVil22arX:Glimm}.
 \end{proof}

 The previous result also works, mutatis-mutandi, with $\sdiv ()$ instead of $\wdiv ()$. In this case, an inspection of \cite[Theorem~3.10]{ThiVil22arX:Glimm} gives the following:
 
 \begin{thm}
  Let $I$ be an ideal of a \ca{} $A$. Assume that there exist $m_0,m_1,n_0,n_1\in\NN$ such that 
  \[
   \sdiv_{2m_0}(m_0 [a])\leq n_0,\andSep 
   \sdiv_{2m_1}(m_1[b])\leq n_1
  \]
  for every $a\in I_+$ and every $b\in (A/I)_+$. Then, there exist  $M,N\in\NN$ depending only on $m_0,m_1,n_0,n_1$ such that 
  \[
   \sdiv_{2M}(M[c])\leq N
  \]
  for every $c\in A_+$.
  
  Conversely, $\sdiv_{2m}(m[a]),\sdiv_{2m}(m[b])\leq n$ for every $a\in (A/I)_+$ and $b\in I_+$ whenever $\sup_{c\in A_+}\sdiv_{2m}(m[c])\leq n$ for some $m,n\in\NN$.
 \end{thm}

  As noted in \autoref{prp:SimpFinDiv}, having a sup-dense subset of elements with finite divisibility does not imply that every element has such property. However, \autoref{prp:FinDivSoftBound} below shows that it is enough to check finite divisibility for strongly soft elements (which, in general, are not sup-dense).
  
  \begin{pgr}[Strongly soft elements and retracts]
   Let $A$ be a \ca{}. As defined in \cite{ThiVil22pre:Soft}, we denote by $\Cu (A)_\soft$ the set of \emph{strongly soft elements}, that is, those elements $x\in\Cu (A)$ such that for every $x'\ll x$ there exists $t\in\Cu (A)$ satisfying $x'+t\leq x\leq \infty t$. When $A$ is residually stably finite, $[a]\in\Cu (A)$ is strongly soft if and only if $\overline{aAa}$ has no nonzero, unital quotients; see \cite[Proposition~4.16]{ThiVil22pre:Soft}.
   
   Under certain assumptions, given any $x\in\Cu (A)$ one can find the largest strongly soft element below $x$. When this is the case, we denote by $\sigma\colon \Cu (A)\to \Cu (A)_\soft$ the map that sends an element $x$ to the largest strongly soft element dominated by $x$.
   
   As shown in \cite[Proposition~2.9]{Thi20RksOps}, this map can be defined whenever $A$ is separable, simple, and of stable rank one. In fact, one can show that $\sigma$ is an order- and suprema-preserving monoid morphism.
   
   In \cite[Theorem~5.6]{AsaVasThiVil23arX:DimRadCompSofCuSgp}, it is shown that $\sigma$ can always be defined whenever $A$ is separable, has the Global Glimm Property, and its Cuntz semigroup is \emph{left-soft separative}, that is, if for any triplet of elements $y,t\in S$ and $x\in S_\soft$ satisfying 
    \[
      x+t\ll y+t,\quad 
      t\ll \infty y,\andSep 
      t\ll \infty x,
    \]
    we have $x\ll y$; see \cite[Definition~3.2]{AsaVasThiVil23arX:DimRadCompSofCuSgp}. As explained in \cite[Section~3]{AsaVasThiVil23arX:DimRadCompSofCuSgp}, \ca{s} with strict comparison or stable rank one have a left-soft separative Cuntz semigroup.
    
    In this case, $\sigma$ is an order- and suprema-preserving superadditive map that satisfies $x\leq \sigma (x)+t$ whenever $x\leq \infty t$.
  \end{pgr}
  
  
  \begin{prp}\label{prp:FinDivSoftBound}
   Let $A$ be a separable \ca{} with the Global Glimm Property such that $\Cu (A)$ is left-soft separative, and let $x\in \Cu (A)$ and $m\in\NN$. Then
   \[
    \wdiv_m(\sigma (x))-1\leq   \wdiv_m(x)\leq \wdiv_m(\sigma (x))+1.
   \]
  \end{prp}
  \begin{proof}
   First, to see that $\wdiv_m(\sigma (x))\leq   \wdiv_m(x)+1$, assume that $\wdiv_m(x)=n$ for some $n\in\NN$, since otherwise there is nothing to prove.
   
   Let $s\ll \sigma (x)$. Using that $\sigma$ preserves suprema of increasing sequences, we can find $x'\ll x$ such that $s\ll \sigma (x')$.
   
   By \cite[Proposition 7.7]{ThiVil22pre:Soft}, there exists $t\in\Cu (A)_\soft$ such that $(nm)t\leq x\leq \infty t$. Note that $(nm)t\leq\sigma (x)$ because $t$ is strongly soft. Using that $x$ is weakly $(m,n)$-divisible, one finds $z_1,\ldots ,z_n\in\Cu (A)$ such that 
   \[
    \sigma (x')\leq z_1+\ldots +z_n,\andSep 
    mz_j\leq x
   \]
   for every $j$.
   
   Now, using that $z_i\leq\infty t$ for every $i$, we obtain 
   \[
    \sigma (x')\leq z_1+\ldots +z_n
    \leq (\sigma (z_1)+t)+z_2+\ldots +z_n
    \leq \sigma (z_1)+\ldots +\sigma (z_n)+nt.
   \]

   Further, using that $\sigma$ is superadditive, we also get $m\sigma (z_j)\leq \sigma (mz_j)\leq \sigma (x)$. Since $n(mt)\leq \sigma (x)$, we deduce that $\wdiv_n(\sigma (x))\leq   n+1$, as desired.
   
   To prove that $\wdiv_m(x)\leq \wdiv_m(\sigma (x))+1$, take $t\in\Cu (A)$ such that $mt\leq x\leq \infty t$. Note that this can be done because $A$ has the Global Glimm Property and, consequently, $\Cu (A)$ is $(m,\omega )$-divisible; see \autoref{pgr:2OmegaDiv}. Thus, we have $x\leq \sigma (x)+t$. Let $x'\ll x$. Since $x'\ll \sigma (x)+t$, we can use \axiomO{6} to obtain an element $y$ such that $x'\ll y+t$ with $y\ll \sigma (x)$. 
   
   Let $n$ be such that $\sigma (x)$ is weakly $(m,n)$-divisible. Then, there exist $z_1,\ldots ,z_n$ satisfying 
   \[
    y\leq z_1+\ldots +z_n,\andSep 
    mz_j\leq \sigma (x)
   \]
   for each $j$.
   
   We obtain 
   \[
    x'\leq z_1+\ldots +z_n +t,\quad 
    mt\leq x,\andSep 
    mz_j\leq \sigma (x)\leq x
   \]
   for every $j$, which implies that $\wdiv_m(x)\leq n+1$.
  \end{proof}
  
\begin{lma}\label{prp:InfSoft}
 Let $A$ be a \ca{}, and let $s,t\in \Cu (A)$. Assume that $s$ is strongly soft. Then, the element $s\wedge \infty t$ is strongly soft in $\Cu (A)$.
\end{lma}
\begin{proof}
 Let $x',x\in\Cu (A)$ be such that $x'\ll x\ll s\wedge \infty t$. By definition, this implies that $x'\ll s$ and $x'\ll \infty t$.
 
 Since $s$ is strongly soft, we find $y\in\Cu (A)$ such that $x'+y\leq s\leq \infty y$. Thus, using that taking infima with an idempotent is a monoid morphism (\cite[Section~2]{AntPerRobThi21Edwards}), we have 
 \[
  x'+y\wedge \infty t = (x'+y)\wedge \infty t\leq s\wedge\infty t\leq \infty y\wedge \infty t = \infty(y\wedge \infty t).
 \]

 This shows that $s\wedge \infty t$ is strongly soft.
\end{proof}


When $\Cu (A)$ is not left-soft separative, we still have the following:
\begin{thm}\label{prp:BoundDivSoft}
 Let $A$ be a \ca{} with the Global Glimm Property. Then, the following are equivalent:
 \begin{itemize}
  \item[(i)] there exist $n,m\in\NN$ such that $\wdiv_{2m} (m[a])\leq n$ for every $[a]\in \Cu (A)$;
  \item[(ii)] there exist $n,m\in\NN$ such that $\wdiv_{2m} (m[a])\leq n$ for every $[a]\in \Cu (A)_\soft $.
 \end{itemize}
\end{thm}
\begin{proof}
 That (i) implies (ii) is trivial.

 To prove the converse, let $x',x\in\Cu (A)$ be such that $x'\ll x$. It follows from \cite[Proposition~7.7]{ThiVil22pre:Soft} that there exists $y\in\Cu (A)_\soft$ satisfying  $y\leq x\leq \infty y$. In particular, one has $x'\ll \infty y$.
 
 By \cite[Proposition~5.6]{ThiVil22pre:Soft}, we can take $y',y''\in \Cu (A)_\soft$ such that $y'\ll y''\ll y$ and $x'\ll \infty y'$. Since $y''$ is soft, we find $t\in \Cu (A)_\soft$ such that $y'+t\leq y''\leq \infty t$; see \cite[Proposition~4.13]{ThiVil22pre:Soft}. Now, by \axiomO{5}, there exists $c\in\Cu (A)$ satisfying 
 \[
  y''+c\leq x\leq y+c.
 \]

 Thus, we have
 \[
  y'+(c+t)\leq x\leq y+c\leq y+ (c+t).
 \]
 
 Using that $x'\leq \infty y'$, we obtain 
 \[
  x'\leq x\wedge \infty y''\leq (y+ (c+t))\wedge \infty y''= y\wedge \infty y'' +(c+t)\wedge \infty y''
 \]
 and 
 \[
  y\wedge \infty y'', (c+t)\wedge \infty y''\leq x.
 \]

 By \autoref{prp:InfSoft}, the elements $y\wedge \infty y''$ and $t\wedge \infty y''$ are strongly soft. Further, note that one gets 
 \[
  (c+t)\wedge \infty y'' = c\wedge \infty y'' + t\wedge \infty y''
 \]
 with $c\wedge \infty y''\leq \infty y''\leq \infty (t\wedge \infty y'')$.
 
 Using \cite[Theorem~4.16]{ThiVil22pre:Soft}, we deduce that $(c+t)\wedge \infty y''$ is also strongly soft. In other words, for every pair of elements $x',x\in \Cu (A)$ such that $x'\ll x$, we find $s_1,s_2\in\Cu (A)_\soft$ such that 
 \[
  s_1 ,s_2\leq x,\andSep x'\leq s_1+s_2.
 \]
 
 In particular, $ms_1,ms_2\leq mx$ and $mx'\leq ms_1+ms_2$.

 By \autoref{prp:Weak2wSum}~(ii), we have 
 \[
  \wdiv_{2m}(mx)\leq \wdiv_{2m}(ms_1 ) + \wdiv_{2m}(ms_2 )
  \leq 2n,
 \]
 as desired.
\end{proof}

\begin{cor}\label{prp:RetractBound}
 Let $A$ be a \ca{} with the Global Glimm Property. Assume that $\Cu (A)_\soft$ is a retract of a \CuSgp{} $S$ satisfying $\sup_{s\in S}\wdiv_{2k} (ks)<\infty$ for some $k\in\NN$. Then, there exist $n,m\in\NN$ such that $\wdiv_{2m} (m[a])\leq n$ for every $[a]\in\Cu (A)$.
\end{cor}
\begin{proof}
 Using that $\Cu (A)_\soft$ is a retract of $S$, it follows that 
 \[
  \wdiv_{2k} (ky)\leq \sup_{s\in S}\wdiv_{2k} (ks)<\infty
 \]
 for every $y \in \Cu (A)_\soft$.
 
 Applying \autoref{prp:BoundDivSoft}, there exist $n,m\in\NN$ such that $\wdiv_{2m}(m x)\leq n$ for every $x\in\Cu (A)$.
\end{proof}

Recall from \cite[Definition~3.1]{ThiVil22DimCu} that one says that $\Cu (A)$ has \emph{covering dimension zero} if, whenever $x'\ll x\ll y_1+y_2$, there exist $z_1,z_2\in\Cu (A)$ such that $x'\leq z_1+z_2\leq x$. Examples of dimension zero include the Cuntz semigroup of any real rank zero \ca{}, as well as other semigroups, such as the Cuntz semigroup of the Jacelon-Razak algebra; see \cite[Section~5]{ThiVil22DimCu}.

We say that a Cuntz semigroup is \emph{weakly cancellative} if $x\ll y$ whenever $x+z\ll y+z$ for some element $z$. Stable rank one \ca{s} have a weakly cancellative Cuntz semigroup by \cite[Theorem~4.3]{RorWin10ZRevisited}.

\begin{cor}\label{prp:Dim0BoundDiv}
 Let $A$ be a simple, non-elementary \ca{} of Cuntz covering dimension zero. Assume that $\Cu (A)$ is weakly cancellative. Then there exists $n\in\NN$ such that $\wdiv_{2} ([a])\leq n$ for every $[a]\in\Cu (A)$.
\end{cor}
\begin{proof}
 Assume first that $A$ is separable. \cite[Theorem 7.10]{ThiVil22DimCu} shows that $\Cu (A)_\soft$ is a retract of an almost divisible \CuSgp{} $S$; see also the proof of \cite[Proposition~7.13]{ThiVil22DimCu}. 
 
 As explained in \autoref{exa:exaDivRR0}, almost divisibility implies $\wdiv_2(s)\leq 3$ for every element $s\in S$. Thus, the result in this case follows from \autoref{prp:RetractBound} above.
 
 If $A$ is not separable, one can use the same techniques as in the proof of \autoref{prp:SurjMap} to get the desired result.
\end{proof}
 
\section{Nowhere scattered corona and multiplier algebras}\label{sec:NScaCorAlg}

In this section we study when a mutilplier algebra of a $\sigma$-unital, nowhere scattered \ca{} $A$ is nowhere scattered. Our main result is that this happens whenever $\sup_{a\in A_+}\wdiv_{2m}(m[a])<\infty$ for some $m\in\NN$; see \autoref{thm:MainMA}. Note that the supremum is not taken over all the elements of the Cuntz semigroup (ie. over $A\otimes \mathcal{K}$), but only over those with a representative in $A_+$. This set is called the \emph{scale} of the Cuntz semigroup in \cite[4.2]{AntPerThi20CuntzUltraproducts}.

As an application, we prove in \autoref{prp:FinNucDimMANSca} that nowhere scattered \ca{s} of finite nuclear dimension, or real rank zero, or stable rank one and $k$-comparison all have a nowhere scattered multiplier algebra. We also show in \autoref{prp:MAStable} that, for stable \ca{s}, the multiplier algebra being nowhere scattered implies that every element in the algebra has finite divisibility.

Let us begin the section with some examples:

\begin{exa}
 Let $A$ be a weakly purely infinite \ca{}. It was shown in \cite[Proposition~4.11]{KirRor02InfNonSimpleCalgAbsOInfty} that $\mathcal{M}(A)$ is weakly purely infinite, and thus nowhere scattered by \cite[Example~3.3]{ThiVil21arX:NowhereScattered}.
 
 As a similar example, it follows from \cite{LinSimpleCorAlg} that a simple, $\sigma$-unital, non-elementary (ie. nowhere scattered) \ca{} with continuous scale has a purely infinite corona algebra. This implies that $\mathcal{M}(A)$ is nowhere scattered by \cite[Proposition 4.2]{ThiVil21arX:NowhereScattered}; see also \cite{KafNgZha19PurInfMultAlg} and \cite{LinSimplePurInfCorAlg}.
\end{exa}

The examples below appeared in \cite[Examples~4.14,~4.15]{ThiVil21arX:NowhereScattered}. We recall them here for the convenience of the reader.

\begin{exa}\label{exa:product}
 Let $A_k$ be the family of separable, simple, AH-algebras from \cite{RobRor13Divisibility} (see \autoref{exa:DirSumInfDiv}). As shown in \cite[Corollary~8.6]{RobRor13Divisibility}, the product $\prod_k A_k$ has a one-dimensional, irreducible representation. By \cite[Theorem~3.1]{ThiVil21arX:NowhereScattered}, this implies that $\prod_k A_k$ is not nowhere scattered.  Consequently, $A:=\oplus A_k$ is a nowhere scattered \ca{} with a multiplier algebra $\mathcal{M}(A)\cong \prod_k A_k$ that is not nowhere scattered.
\end{exa}

\begin{exa}\label{exa:simple}
 In \cite[Theorem~1]{Sak71DerivationsSimple3}, Sakai constructs a simple \ca{} $A$ such that its derived algebra $D(A)$ satisfies $D(A)/A\cong\CC$. Using that $D(A)\cong \mathcal{M}(A)$ for simple \ca{s} (see the remarks after \cite[Proposition~2.6]{Ped72ApplWeakSemicontinuity}), one sees that $\mathcal{M}(A)$ cannot be nowhere scattered.
 
 It will follow from \autoref{thm:MainMA} and \autoref{prp:FNucDimImpMWDiv} that $A$ is of infinite nuclear dimension. \cite[Corollary~1]{Sak71DerivationsSimple3} already showed that $A$ is not $\mathcal{Z}$-stable.
\end{exa}

The following lemma will play an important role in the proof of \autoref{prp:CoronaNSca}.

\begin{lma}\label{lma:NScaCarac}
 Let $A$ be a \ca{}. The following are equivalent:
 \begin{itemize}
  \item[(i)] $A$ is nowhere scattered.
  \item[(ii)] For every $a\in A_+$ and $n\in\NN$, the class $[a\otimes 1_n]\in\Cu (A)$ is weakly $(2n,\omega )$-divisible.
  \item[(iii)] For every $a\in A_+$, there exists $n=n(a)\in\NN$ such that the class $[a\otimes 1_{n}]\in\Cu (A)$ is weakly $(n+1,\omega )$-divisible.
 \end{itemize}
\end{lma}
\begin{proof}
 We know from \cite[Theorem~8.9]{ThiVil21arX:NowhereScattered} that $A$ is nowhere scattered if and only if every element in $\Cu (A)$ is weakly $(2,\omega )$-divisible. Using \cite[Theorem~8.9]{ThiVil21arX:NowhereScattered}, it follows that every element is $(2n,\omega )$-divisible for every $n\in\NN$. This proves that (i)~implies~(ii), and it is trivial that (ii)~implies~(iii).
 
 To see that (iii) implies (i), assume for the sake of contradiction that $A$ is not nowhere scattered. Then, we know from \cite[Theorem~3.1]{ThiVil21arX:NowhereScattered} that there exist ideals $I\subseteq J\subseteq A$ such that $J/I$ is elementary, that is, $\Cu (I/J)\cong \Cu (\mathbb{C})\cong\NNbar$; see \cite[Theorem~4.4.4]{Eng14PhD}.
 
 Note that, since we are assuming that $A$ satisfies (iii), both $I$ and $I/J$ satisfy (iii) as well. Let $\phi$ be an isomorphism from $\Cu (I/J)$ to $\NNbar$. Now, it follows from \cite[Lemma~3.3]{ThiVil22arX:Glimm} that for every positive element $b\in (I/J)\otimes \mathcal{K}$ there exists $a\in (I/J)_+$ such that $[b]\leq \infty [a]$ in $\Cu (I/J)$. Thus, let $a\in (I/J)_+$ be such that $\phi ([a])\neq 0$. We get $1\ll 1\leq \phi ([a])$ and, using R{\o}rdam's lemma (see, for example, \cite[Theorem~2.30]{Thi17:CuLectureNotes}), there exists $a'\in (I/J)_+$ such that $1=\phi ([a'])$.
 
 Take $n\in\NN$ such that $[a'\otimes 1_n]$ is weakly $(n+1,\omega )$-divisible. Since $\phi$ is an isomorphism, we must have that $\phi ([a'\otimes 1_n])=n$ is also weakly $(n+1,\omega )$-divisible. However, the only element $x\in\NNbar$ such that $(n+1)x\leq n$ is zero. This contradicts the weak $(n+1,\omega )$-divisibility of $[a'\otimes 1_n]$.
 
 Thus, $A$ has no nonzero elementary ideal-quotients, as desired.
\end{proof}

\begin{rmk}\label{rmk:01inft}
 In the proof of `(iii)$\implies$(i)' in \autoref{lma:NScaCarac} above, we have used that a \ca{} is nowhere scattered if and only if it has no elementary ideal-quotients. As shown in \cite[Proposition~8.8]{ThiVil21arX:NowhereScattered}, a \CuSgp{} $S$ satisfying \axiomO{5}-\axiomO{8} has no elementary ideal-quotients if and only if every element in $S$ is weakly $(2,\omega )$-divisible.
 
 In light of this, one might expect `(iii)$\implies$(i)' to hold for every \CuSgp{} $S$ satisfying \axiomO{5}-\axiomO{8}, that is, that every element in $S$ is weakly $(2,\omega )$-divisible whenever for every element $x\in S$ there exists $n_x$ such that $n_x x$ is weakly $(n_x +1,\omega )$-divisible.
 
 However, this is not true: For example, $S=\{ 0,1,\infty \}$ is a \CuSgp{} satisfying \axiomO{5}-\axiomO{8}. $S$ is not weakly $(2,\omega )$-divisible, but every element in $S$ has a  properly infinite multiple. Thus, it satisfies (iii), but not (i).

 The reason for this disparity is due to the fact that, in the context of abstract Cuntz semigroups, $\NNbar$ is not the only elementary \CuSgp{}. In fact, $\{ 0,1,\infty \}$ is elementary; see \cite[Section~8]{ThiVil21arX:NowhereScattered}.
\end{rmk}

\autoref{prp:KirRor} below is \cite[Lemma~4.10]{KirRor02InfNonSimpleCalgAbsOInfty}, which uses some of the ideas from \cite{Ell74:DevMatAlg}. It was stated with the assumption of weak pure infiniteness, but an inspection of their proof shows that this is not actually needed. This was also stated, in a different way and with a different proof, in \cite[Theorem~4.2]{KafNgZha17StrCompMultiplier}.

\begin{lma}\label{prp:KirRor}
 Let $A$ be a $\sigma$-unital \ca{}, let $T\in \mathcal{M}(A)_+$, and let $\varepsilon >0$. Then $A$ has an increasing, countable, approximate unit $(e_n)_n$ of positive contractions with $e_{n+1}e_n=e_n$ such that 
 \[
  T = a + \sum_{n=1}^{\infty} f_{2n-1}^{1/2}Tf_{2n-1}^{1/2} + \sum_{n=1}^{\infty} f_{2n}^{1/2} Tf_{2n}^{1/2},\quad\quad 
  f_n := e_n - e_{n-1}
 \]
 with $e_0=0$ and $a\in A$ satisfying $\Vert a \Vert \leq \varepsilon$.
\end{lma}

In the proof of \autoref{prp:CoronaNSca} we will also need the following:

\begin{lma}\label{prp:DiagSumDiv}
Let $A$ be a \ca{}, and let $a,b\in A_+$ be elements such that $\mwdiv ([a]) , \mwdiv ( [b])<\infty$. Then $\mwdiv ([a+b])<\infty$.

In particular, there exists $n\in\NN$ such that $(a+b)\otimes 1_{n}$ has a $(n+1 ,\omega )$-divisible class.
\end{lma}
\begin{proof}
 Let $n_a,n_b\in\NN$ be such that $\wdiv_{2n_a}(n_a[a]),\wdiv_{2n_b}(n_b[b])<\infty$.

 Using that $n_a n_b[a+b]=[a\otimes 1_{n_a n_b} + b\otimes 1_{n_a n_b}]$, we can apply \autoref{prp:Weak2wSum}~(iii) and get 
 \[
  \wdiv_{2n_a n_b}(n_an_b[a+b]) \leq \wdiv_{2n_a n_b}(n_an_b[a]) + 
  \wdiv_{2n_a n_b}(n_an_b[b]).
 \]

 Thus, applying \autoref{prp:DivStComp} to both summands, we obtain 
 \[
  \wdiv_{2n_a n_b}(n_an_b[a+b]) \leq \wdiv_{2n_a}(n_a[a]) + 
  \wdiv_{2 n_b}(n_b[b])<\infty ,
 \]
 as desired.
\end{proof}


\begin{lma}\label{prp:StConvSumsDiv}
 Let $A$ be a $\sigma$-unital \ca{}, and let $(a_i)_{i=1}^\infty$ be a bounded sequence in $A$ of pairwise orthogonal elements such that $R=\sum_i a_i$ is strictly convergent and $a_i\perp e_{i-1}$ for every $i\geq 1$, where $(e_i)_i$ is an approximate unit as in \autoref{prp:KirRor}.
 
 Let $m\in\NN$.  Then, for every $\varepsilon_0 >0$, one has
 \[
  \wdiv_{2m}(m[R])\leq \sup_{i}\sup_{0\leq \varepsilon \leq \varepsilon_0} \wdiv_{2m}(m[(a_i -\varepsilon )_+])
 \]
 in $\Cu (\mathcal{M}(A))$.
\end{lma}
\begin{proof}
 Let $\varepsilon_0>0$. Assume that $n:=\sup_i\sup_{\varepsilon\leq \varepsilon_0} \wdiv_{2m}(m[(a_i-\varepsilon )_+])$ is finite, since otherwise there is nothing to prove. Fix $\varepsilon \leq\varepsilon_0$ positive. By assumption, $(a_i-\tfrac{\varepsilon}{4})_+\otimes 1_m$ is weakly $(2m,n)$-divisible for each $i$. Thus, we can find elements $b_{i,j}\in A\otimes\mathcal{K}$ for $j=1,\ldots , n$ such that 
\[
 m[(a_i-\varepsilon/3 )_+]\leq [b_{i,1}]+\ldots +[b_{i,n}]
 ,\andSep
 2m[b_{i,j}]\leq m[(a_i-\varepsilon/4 )_+]
\]
for each $i$ and $j$.

It follows from \cite[Lemma~2.3(i)]{RobRor13Divisibility} that there exist elements $c_{i,j}\in M_m(A)_+$ with  $c_{i,j}\precsim b_{i,j}$ such that 
\[
 (a_i - 2\varepsilon/3 )_+\otimes 1_m = \sum_{j=1}^n c_{i,j}.
\]

Set $d_{i,j}:= ((a_i - 2\varepsilon /3 )_+\otimes 1_m) c_{i,j} ((a_i - 2\varepsilon /3)_+\otimes 1_m)\in M_m (A)$. Note that every entry in $d_{i,j}$ is in $\overline{a_i Aa_i}$, and is thus orthogonal with $e_{i-1}$. Further, since the sequence $(d_{i,j})_i$ is bounded for every fixed $j$, so is each of the induced entry-wise sequences. This shows that the (strictly convergent) sum $R_j := \sum_{i} d_{i,j}$ is in $M_m (\mathcal{M}(A))$.

Thus, we have 
\[
 (R-2\varepsilon/3 )_+^3\otimes 1_m = \sum_{j=1}^n R_j
\]
and, therefore, $m[(R-\varepsilon )_+]\ll \sum_{j=1}^n [R_j]$.

Let $\delta >0$ be such that $m[(R-\varepsilon )_+]\leq \sum_{j=1}^n [(R_j -\delta )_+]$. Applying \cite[Lemma~2.2]{KafNgZha17StrCompMultiplier} at 
\[
 d_{i,j}\otimes 1_{2m}\precsim c_{i,j}\otimes 1_{2m}\precsim 
 b_{i,j}\otimes 1_{2m}\precsim (a_i-\varepsilon/4 )_+\otimes 1_m
\]
we find elements $r_{i,j}$ such that 
\[
 (d_{i,j}-\delta)_+\otimes 1_{2m} = r_{i,j}(a_i\otimes 1_m)r_{i,j}^*,\andSep 
 \sup_i \Vert r_{i,j}\Vert^2 \leq \frac{4}{\varepsilon} \sup_i \Vert d_{i,j}\otimes 1_{2m}\Vert.
\]

Set $x_{i,j}:= (a_i\otimes 1_m)^{1/2}r_{i,j}^* (d_{i,j}\otimes 1_{2m}-\delta )_+$. We get
 \[
  x_{i,j}^*x_{i,j}= (d_{i,j}-\delta )_+^3\otimes 1_{2m},\andSep 
  x_{i,j}x_{i,j}^*\in \overline{(a_i\otimes 1_m) M_m(A) (a_i\otimes 1_m)},
 \]
 with both $\left((a_i\otimes 1_m)^{1/2}r_{i,j}^*\right)_i$ and $\left( r_{i,j}^* (d_{i,j}\otimes 1_{2m}-\delta )_+\right)_i$ being bounded sequences for each fixed $j$.
 
 Let $x_{i,j}(s,r)$ denote the $(s,r)$-th entry of $x_{i,j}$. Working as in \cite[Proposition~4.4]{KafNgZha17StrCompMultiplier}, one can now show that the following sums are all strictly convergent
 \[
  x_{j}(s,r):=\sum_i x_{i,j}(s,r).
 \]

Set $x_{j}:=(x_{j}(s,r))_{s,r}\in M_{m,2m^2}(\mathcal{M}(A))$. One gets
\[
 x_j^* x_j = (R_j -\delta)_+^3\otimes 1_{2m}\sim (R_j -\delta)_+\otimes 1_{2m},
 \]
 and 
 \[
 x_j x_j^*\in \overline{(R\otimes 1_m)M_m (\mathcal{M}(A) )(R\otimes 1_m)}.
\]

It follows that, in $\Cu (\mathcal{M}(A))$, we have $2m[(R_j-\delta )_+]\leq m[R]$ for each $j\in\NN$. Since $[R]=\sup_{0\leq \varepsilon\leq \varepsilon_0}[(R-\varepsilon )_+]$, we have that $m[R]$ is weakly $(2m,n)$-divisible, as desired.
\end{proof}

Let us say that a sequence $(x_i)_i$ in a \ca{} $A$ is \emph{orthogonal} if the elements in the sequence are pairwise orthogonal, that is, if $x_i x_j=0$ whenever $i\leq j$.

\begin{prp}\label{prp:CoronaNSca}
 Let $A$ be a $\sigma$-unital \ca{}. Assume that for every orthogonal sequence $(a_i)_i$ of positive elements in $A$ there exist $m,n\in\NN$ such that $\wdiv_{2m} (m[a_i])\leq n$ for each $i$. Then $\wdiv_{2m} (m[T])<\infty$ for every $T\in \mathcal{M}(A)/A$.
 
 In particular, $\mathcal{M}(A)/A$ is nowhere scattered.
\end{prp}
\begin{proof}
Let $\pi\colon \mathcal{M}(A)\to \mathcal{M}(A)/A$ denote the quotient map. Using Lemmas \ref{prp:KirRor} and \ref{prp:StConvSumsDiv}, we deduce that each positive element in the corona algebra can be written as $\pi (R_1)+\pi (R_2)$ with $\mwdiv (R_1), \mwdiv (R_2)<\infty$. 

Using \autoref{prp:DiagSumDiv}, one has $\mwdiv (\pi (T))<\infty$ in $\mathcal{M}(A)/A$ for every element $T\in\mathcal{M}(A)$. \autoref{lma:NScaCarac} shows that the corona algebra is nowhere scattered, as desired.
\end{proof}


%

\begin{rmk}\label{rmk:ExaSimpNotDiv}
 Let $A$ be the simple, nowhere scattered \ca{} from \autoref{exa:simple}. It follows from \autoref{prp:CoronaNSca} above that there must exist an element $a\in (\oplus_{i=1}^\infty A)_+$ such that $\mwdiv ([a])=\infty$.
\end{rmk}

Using \autoref{prp:CoronaNSca}, we can now prove the main result of the paper.

\begin{thm}\label{thm:MainMA}
 Let $A$ be a $\sigma$-unital \ca{}. Assume that for every orthogonal sequence $(a_i)_i$ of positive elements in $A$ there exist $m,n\in\NN$ such that $\wdiv_{2m} (m[a_i])\leq n$ for each $i$. Then $\mathcal{M}(A)$ is nowhere scattered.
\end{thm}
\begin{proof}
 Using that nowhere scatteredness works well with extensions (\cite[Proposition 4.2]{ThiVil21arX:NowhereScattered}), the result is a direct consequence of \autoref{prp:CoronaNSca}.
\end{proof}

It follows from the results in \cite{PerRor04AFembeddings} that the multiplier algebra of a simple, separable \ca{} of real rank zero does not admit a character. \autoref{prp:FinNucDimMANSca}~(i) below generalizes this result.

\begin{thm}\label{prp:FinNucDimMANSca}
 Let $A$ be a $\sigma$-unital \ca{}. Assume that $A$ has 
 \begin{itemize}
  \item[(i)] real rank zero, or
  \item[(ii)] finite nuclear dimension, or
  \item[(iii)] $k$-comparison and a surjective rank map.
 \end{itemize}
 
 Then $A$ is nowhere scattered if and only if $\mathcal{M}(A)$ is nowhere scattered.
\end{thm}
\begin{proof}
 Assume that $\mathcal{M}(A)$ is nowhere scattered. \cite[Proposition 4.2]{ThiVil21arX:NowhereScattered} implies that $A$ is nowhere scattered as well.

 Conversely, assume that $A$ is nowhere scattered. It follows from \autoref{exa:exaDivRR0}, \autoref{prp:FNucDimImpMWDiv} and \autoref{prp:SurjMap} respectively that (i), (ii) and (iii) imply that every sequence of positive elements (orthogonal or not) in $A_+$ has uniformly bounded multiple divisibility. Thus, we can use \autoref{thm:MainMA} to get the desired result.
\end{proof}

\begin{qst}\label{qst:NSCaFinmdiv}
 Let $A$ be a nowhere scattered \ca{} such that $\mwdiv ([a])<\infty$ for every $a\in A_+$. Does it follow that $\mathcal{M}(A)$ is nowhere scattered?
 
 Note that this would imply that the simple \ca{} from \autoref{exa:simple} contains a $(2,\omega )$-divisible element of infinite divisibility.
\end{qst}

\begin{rmk}
 The assumption from \autoref{thm:MainMA} above is not equivalent to nowhere scatteredness, not even in the simple or $\sigma$-unital case:
 
 \begin{itemize} 
  \item[(1)] The example from \autoref{exa:product} is $\sigma$-unital. Thus, there exist $\sigma$-unital, nowhere scattered \ca{s} that do not satisfy the condition from \autoref{thm:MainMA}.
  
  \item[(2)] By \autoref{prp:SimpFinDiv}, the set of classes with finite divisibility is sup-dense in the Cuntz semigroup of a simple, non-elementary \ca{} $A$. However, and as made explicit in \autoref{exa:simple}, this is not enough for the multiplier algebra $\mathcal{M}(A)$ to be nowhere scattered.
 \end{itemize}
 
 Further, note that if the converse of \autoref{thm:MainMA} holds for some family of \ca{s}, one must have that any unital, nowhere scattered \ca{} $A$ in the family must satisfy $\sup_{a\in A_+}\wdiv_{2m} (m[a_i])<\infty$ for some $m$.
 \end{rmk}
 
 As another application of \autoref{prp:StConvSumsDiv}, we can study when multiplier algebras have a character:
 
\begin{prp}\label{prp:CharacMA}
 Let $A$ be a $\sigma$-unital \ca{}, and let $(e_i)_i$ be an approximate unit as in \autoref{prp:KirRor}. Assume that 
 \[
  \sup_i\sup_{0\leq \varepsilon \leq \varepsilon_0} \wdiv_{2m} (m[(e_i-e_{i-1}-\varepsilon)_+])<\infty
 \]
 for some $\varepsilon_0 >0$. 
 
 Then, $\mathcal{M}(A)$ does not admit a character.
\end{prp}
\begin{proof}
 Using \autoref{prp:StConvSumsDiv}, we see that the unit $1=\sum_i (e_i - e_{i-1})$ in $\mathcal{M}(A)$ satisfies $\wdiv_{2m} (m[1])<\infty$.
 
 The result now follows from \cite[Corollary~5.6]{RobRor13Divisibility}.
\end{proof}

In particular, \autoref{prp:CharacMA} recovers \cite[Corollary~8.5]{RobRor13Divisibility}:

\begin{cor}\label{prp:CharacCharac}
 The product of unital \ca{s} $A_i$ does not admit a character whenever the supremum of $\wdiv_2 ([1_i])$ is finite.
\end{cor}

 \subsection{Stable \ca{s}}
 
 The multiplier algebra of a $\sigma$-unital, stable \ca{} $A$ has been studied extensively. For example, \cite{Ror91IdMult} studies when the corona algebra of $A$ is simple, while in \cite{WinNgNucDimCorFacProp} it is shown, using its multiplier algebra, that a nowhere scattered \ca{} of finite nuclear dimension has the corona factorization property.
 
 In what follows, we show that $\mwdiv ([a])$ must be finite for every $a\in A_+$ if $\mathcal{M}(A)$ is to be nowhere scattered; see \autoref{prp:MAStable}. Thus, if \autoref{qst:NSCaFinmdiv} has an affirmative answer, this would imply that the multiplier algebra of a $\sigma$-unital, stable \ca{} is nowhere scattered if and only if every element in the algebra has a multiple of finite divisibility.
 
\autoref{lma:WeakDivExa} below was stated for $(m,\omega )$-divisibility in \cite{ThiVil22arX:Glimm}. We now provide the analog statement for weak  $(m,\omega )$-divisibility, which only needs to assume the element (and not the whole Cuntz semigroup) to be divisible.

\begin{lma}\label{lma:WeakDivExa}
 Let $A$ be a \ca{}, and let $a\in (A\otimes \mathcal{K})_+$. Then, if $[a]$ is weakly $(m,\omega )$-divisible, there exists a sequence $(y_n)_n\subseteq \Cu (A)$ such that $my_j\leq [a]\leq \sum_{j=1}^\infty y_j$ for every $j\in\NN$.
\end{lma}
\begin{proof}
 Let $\varepsilon_n$ be a strictly decreasing sequence converging to $0$, and let $x_n := [(a-\varepsilon_n )_+]$. We will inductively find elements $y_{k,j} ,c_n$ such that 
 \[
  x_n\leq \sum_{k=1}^{n}\sum_{j}y_{k,j},\andSep 
  my_{k,j}\leq [a]. 
 \]
 
 First, use weak $(m,\omega )$-divisibility to obtain $y_{1,j}$ such that $x_1\leq \sum_j y_{1,j}$ and $my_{1,j}\leq [a]$.
 
 Assume now that we have proven the result for every $k\leq n$. Using \axiomO{5} at $x_{n-1}\ll x_n\leq [a]$, we obtain $c\in\Cu (A)$ such that $x_{n-1}+c\leq [a]\leq x_n +c$. In particular, one has $x_{n+1}\ll [a]\leq x_n +c$. Take $c'\ll c$ such that $x_{n+1}\leq x_n + c'$.
 
 Using $(m,\omega )$-divisibility at $c'\ll c\leq x$, we obtain elements $y_{n+1,j}$ such that $c'\leq \sum_j y_{n+1,j}$ and $my_{n+1,j}\leq x$ for each $j$. We have 
 \[
  x_{n+1}\leq x_n + c'\leq \sum_{k=1}^{n}\sum_{j}y_{k,j} + \sum_j y_{n+1,j} = 
  \sum_{k=1}^{n+1}\sum_{j}y_{k,j},
 \]
 as desired.
 
 By taking the supremum on $n$, one sees that the sequence $(y_{n,j})_{n,j}$ satisfies the desired conditions.
\end{proof}
 
 Note that, in general, one cannot adapt the previous proof to show that there exist $y_1,\ldots ,y_m$ with  $2y_j\leq [a]\leq y_1+\ldots +y_m$ whenever $[a]$ is $(2,m)$-divisible. 
 
 However, this is the case for $\sigma$-unital \ca{s} whose multiplier algebra is nowhere scattered: 
 
 \begin{thm}\label{prp:MAStable}
  Let $A$ be a $\sigma$-unital, stable \ca{}. Assume that $\mathcal{M}(A)$ is nowhere scattered. Then, for every $a\in A_+$ and $m\in\NN$ there exist finitely many elements $y_1,\ldots ,y_n\in\Cu (A)$ such that 
  \[
   my_j\leq [a]\leq y_1+\ldots +y_n
  \] 
  for every $j\leq n$.
  
  In particular, $\mwdiv (a)<\infty$ for every $a\in A_+$.
 \end{thm}
 \begin{proof}
  Let $a\in A_+$ and $m\in\NN$. By \cite[Proposition~2.8]{AntPerRobThi18arX:CuntzSR1} there exists a projection $p_a\in \mathcal{M}(A)$ such that, for every $x\in\Cu (A)$, one has $x\leq [a]$ if and only if $x\leq [p_a]$ in $\Cu (\mathcal{M}(A))$.
  
  Assume that $\mathcal{M}(A)$ is nowhere scattered. Thus, $\Cu (\mathcal{M}(A))$ is weakly $(m,\omega )$-divisible by \cite[Theorem~8.9]{ThiVil21arX:NowhereScattered}. It follows from \autoref{rmk:FinDivProj} that $\wdiv_m ([p])<\infty$ for every projection $p\in \mathcal{M}(A)$.
  
  Then, since we have $[a]\leq [p_a]\ll [p_a]$, there exist elements $z_1,\ldots ,z_n\in \Cu (\mathcal{M}(A))$ such that 
  \[
   [a]\leq [p_a]\leq z_1+\ldots +z_n,\andSep mz_j\leq [p_a]
  \]
 for each $j\leq n$.
 
 Since $\mathcal{M}(A)$ is a \ca{}, we know from \cite[Remark~2.6]{AntPerRobThi21Edwards} that the infimum $y_j:=z_j\wedge \infty [a]$ exists for each $j$. Any representative of this element is contained in the ideal generated by $a$ and, therefore, must be a positive element in $A$. Thus, we have 
 $[a]\leq y_1+\ldots +y_n$ in $\Cu (\mathcal{M}(A))$ with $y_j\in \Cu (A)$ for each $j$.
 
 Further, we also get $my_j=m(z_j\wedge \infty [a])\leq [p_a]$. Thus, $my_j\leq [a]$ in $\Cu (\mathcal{M}(A))$.
 
 Finally, note that, since $A$ is an ideal of $\mathcal{M}(A)$, a pair of elements in $A$ are Cuntz subequivalent in $\mathcal{M}(A)$ if and only if they are Cuntz subequivalent in $A$; see, for example, \cite[Proposition~2.18]{Thi17:CuLectureNotes}.
 
 Consequently, we have 
 \[
  [a]\leq y_1+\ldots +y_n,\andSep 
  my_j\leq [a]
 \]
 in $\Cu (A)$, as desired.
 \end{proof}

\begin{rmk}
 Note that there exist stable, nowhere scattered \ca{s} with a multiplier algebra that is not nowhere scattered. Indeed, simply take $A$ as in \autoref{exa:DirSumInfDiv}. Then, $A\otimes \mathcal{K}$ is a nowhere scattered \ca{} by \cite[Proposition 4.12]{ThiVil21arX:NowhereScattered} that contains a $(2,\omega )$-divisible element of infinite divisibility.
 
 It follows from \autoref{prp:MAStable} that $\mathcal{M}(A\otimes \mathcal{K})$ cannot be nowhere scattered.
\end{rmk}

In this paper we have studied conditions under which a multiplier algebra $\mathcal{M}(A)$ is nowhere scattered. A problem that seems to be much more involved is to determine when $\mathcal{M}(A)$ has the Global Glimm Property.

\begin{qst}
 Assume that $A$ has the Global Glimm Property. When does $\mathcal{M}(A)$ also have the Global Glimm Property?
\end{qst}

Note that this has a positive solution whenever $\sup_{a\in A_+}\sdiv_{2m}(m[a])<\infty$ and $\mathcal{M}(A)/A$ satisfies a condition that ensures that the Global Glimm Problem can be answered affirmatively. For example, Lin and Ng show in \cite{LinNg16CoronaStableZ} that $\mathcal{M}(\mathcal{Z}\otimes\mathcal{K})/\mathcal{Z}\otimes\mathcal{K}$ has real rank zero. Thus, since $\mathcal{Z}\otimes\mathcal{K}$ is covered by \autoref{thm:MainMA}, we get that $\mathcal{M}(\mathcal{Z}\otimes\mathcal{K})/\mathcal{Z}\otimes\mathcal{K}$ is nowhere scattered. By \cite{EllRor06Perturb} (see also \cite[Proposition~7.4]{ThiVil22arX:Glimm}), $\mathcal{M}(\mathcal{Z}\otimes\mathcal{K})/\mathcal{Z}\otimes\mathcal{K}$ has the Global Glimm Property. Consequently, $\mathcal{M}(\mathcal{Z}\otimes\mathcal{K})$ also has the Global Glimm Property by \cite[Theorem~3.10]{ThiVil22arX:Glimm}.

\providecommand{\etalchar}[1]{$^{#1}$}
\providecommand{\bysame}{\leavevmode\hbox to3em{\hrulefill}\thinspace}
\providecommand{\noopsort}[1]{}
\providecommand{\mr}[1]{\href{http://www.ams.org/mathscinet-getitem?mr=#1}{MR~#1}}
\providecommand{\zbl}[1]{\href{http://www.zentralblatt-math.org/zmath/en/search/?q=an:#1}{Zbl~#1}}
\providecommand{\jfm}[1]{\href{http://www.emis.de/cgi-bin/JFM-item?#1}{JFM~#1}}
\providecommand{\arxiv}[1]{\href{http://www.arxiv.org/abs/#1}{arXiv~#1}}
\providecommand{\doi}[1]{\url{http://dx.doi.org/#1}}
\providecommand{\MR}{\relax\ifhmode\unskip\space\fi MR }
\providecommand{\MRhref}[2]{%
  \href{http://www.ams.org/mathscinet-getitem?mr=#1}{#2}
}
\providecommand{\href}[2]{#2}

\end{document}